\documentclass[]{interact}

\usepackage{epstopdf}
\usepackage[caption=false]{subfig}
\usepackage{enumerate}
\usepackage[numbers,sort&compress]{natbib}
\bibpunct[, ]{[}{]}{,}{n}{,}{,}
\makeatletter
\def\NAT@def@citea{\def\@citea{\NAT@separator}}
\makeatother

\theoremstyle{plain}
\newtheorem{theorem}{Theorem}[section]
\newtheorem{lemma}[theorem]{Lemma}
\newtheorem{corollary}[theorem]{Corollary}
\newtheorem{proposition}[theorem]{Proposition}

\theoremstyle{definition}
\newtheorem{definition}[theorem]{Definition}
\newtheorem{example}[theorem]{Example}

\theoremstyle{remark}
\newtheorem{remark}[theorem]{Remark}

\begin{document}


\title{$L^0$--convex compactness and its applications to random convex optimization and random variational inequalities}

\author{
\name{Tiexin Guo\textsuperscript{a}\thanks{CONTACT Email: tiexinguo@csu.edu.cn},Erxin Zhang\textsuperscript{a},Yachao Wang\textsuperscript{a} and Mingzhi Wu\textsuperscript{b}}
\affil{\textsuperscript{a}School of Mathematics and Statistics, Central South University, ChangSha 410083, China; \textsuperscript{b}School of Mathematics and Physics, China University of Geosciences, WuHan 430074, China.}
}

\maketitle

\begin{abstract}
First, this paper introduces the notion of $L^0$--convex compactness for a special class of closed convex subsets--closed $L^0$--convex subsets of a Hausdorff topological module over the topological algebra $L^0(\mathcal{F},K)$, where $L^0(\mathcal{F},K)$ is the algebra of equivalence classes of random variables from a probability space $(\Omega,\mathcal{F},P)$ to the scalar field $K$ of real numbers or complex numbers, endowed with the topology of convergence in probability. Then, this paper continues to develop the theory of $L^0$--convex compactness by establishing various kinds of characterization theorems  on $L^0$--convex compactness for $L^0$--convex subsets of a class of important topological modules--complete random normed modules, in particular, we make full use of the theory of random conjugate spaces to establish the characterization theorem of James type on $L^0$--convex compactness for a closed $L^0$--convex subset of a complete random normed module, which also surprisingly implies that our notion of $L^0$--convex compactness coincides with Gordan \v{Z}itkovi\'{c}'s notion of convex compactness in the context of a closed $L^0$--convex subset of a complete random normed module. As the first application of our results, we give a fundamental theorem on random convex optimization (or, $L^0$--convex optimization), which includes Hansen and Richard's famous result as a special case. As the second application, we give an existence theorem of solutions of random variational inequalities, which generalizes H.Brezis' classical result from a reflexive Banach space to a random reflexive complete random normed module. It should be emphasized that a new method, namely the $L^0$--convex compactness method, is presented for the second application since the usual weak compactness method is no longer applicable in the present case. Besides, our fundamental theorem on random convex optimization can be also applied in the study of optimization problems of conditional convex risk measures, which will be given in our future papers.

\end{abstract}

\begin{keywords}
Complete random normed module; closed $L^0$--convex subsets; $L^0$--convex compactness; random convex optimization; random variational inequalities
\end{keywords}

\section{Introduction}

It is well known that weak compactness is enough to ensure the existence of solutions of convex optimization problem, just as the following fundamental theorem on convex optimization shows:
\begin{theorem}\cite{ET99}\label{Theorem 1.1}
Let $C$ be a closed convex subset of a Hausdorff locally convex space $(E,\mathcal{T})$ and $f : C \to (-\infty,+\infty]$ a proper lower semicontinuous convex function. Then the following hold:
\begin{enumerate}[(1)]
\item If $C$ is weakly compact, then there exists $x_0 \in C$ such that $f(x_0) = \inf_{x \in C} f(x)$; in particular if $E$ is a reflexive Banach space and $f$ is coercive (namely, $f(x_n) \to +\infty$ whenever $\{ x_n : n \in N \}$ is a sequence in $C$ such that $\|x_n\| \to +\infty$), then such an $x_0$ always exists.
\item If $f$ is strictly convex, then there exists at most one $x_0$ in $C$ such that $f(x_0) = \inf_{x \in C} f(x)$.

\end{enumerate}
\end{theorem}

Similarly, weak compactness also played a crucial role in the existence of solutions of variational inequalities in a reflexive Banach space, see \cite{Bre68,ET99,Yuan98,Yuan99}.

Finance and economics force people to consider the convex optimization problem on a not locally convex space. For example, let $L^0(\mathcal{F})$ be the linear topological space of equivalence classes of real--valued random variables on a probability space $(\Omega, \mathcal{F}, P)$, endowed with the usual topology of convergence in probability space, it is typically a not locally convex space, whose dual is trivial when $\mathcal{F}$ is atomless. Clearly, it makes no sense to speak of weak compactness for a closed convex subset of a not locally convex space like $L^0(\mathcal{F})$. To overcome the difficulty, Gordan \v{Z}itkovi\'{c} introduced the following elegant notion of convex compactness for a convex subset of a Hausdorff linear topological space in \cite{Zit10}:
\begin{definition}\cite{Zit10}\label{Definition 1.2}
A nonempty convex subset $C$ of a Hausdorff linear topological space $E$ is said to be convexly compact (or to have convex compactness) if each family of closed convex subsets of $C$ has a nonempty intersection whenever the family has the finite intersection property.
\end{definition}

The notion of convex compactness was employed in \cite{Zit10} to give many successful applications to both nonlinear analysis and mathematical economics. For example, Gordan \v{Z}itkovi\'{c} \cite{Zit10} proved that if $C$ is a closed convexly compact subset of a Hausdorff linear topological space $E$ and $f : C \to (-\infty,+\infty]$ a proper lower semicontinuous convex function, then $f$ attains its minimum over $C$, which, combining the classical James theorem \cite{J64}, shows that a closed convex subset of a complete Hausdorff locally convex space is convexly compact if and only if it is weakly compact. Thus the notion of  Gordan \v{Z}itkovi\'{c}'s convex compactness is a proper substitute for the notion of ``weak compactness" for a closed convex subset of a not locally convex space, in fact, Gordan \v{Z}itkovi\'{c}'s result stated above essentially generalizes Theorem \ref{Theorem 1.1} to a large class of not locally convex spaces. In particular, it was proved in \cite{Zit10} that a convex subset $C$ of $L^0_+(\mathcal{F}) := \{ x \in L^0(\mathcal{F}) : x \geq 0 \}$ is convexly compact if and only if $C$ is bounded in the sense of linear topology on $L^0(\mathcal{F})$, see \cite{Zit10} for more rich other results and applications.

Recently, random functional analysis and its applications to conditional risk measures naturally lead us to study the problem of random convex optimization and random variational inequalities. To let the reader have a good understanding on this problem, we first give a brief introduction of the closely related theoretical and financial backgrounds. Random functional analysis is functional analysis based on random metric spaces, random normed modules, random inner product modules and random locally convex modules, which are a random generalization of ordinary metric spaces, normed spaces, inner product spaces and locally convex spaces, respectively. The history of random functional analysis will unavoidably dates back to the theory of probabilistic metric spaces, which was initiated by K.Menger in 1942 and subsequently founded by B.Schweizer and A.Sklar, see \cite{SS8305} for details. The theory of probabilistic metric spaces is centered at the study of probabilistic metric spaces and probabilistic normed spaces, whose main idea is to use probability distribution functions to describe the probabilistic metric between two points or the probabilistic norm of a vector. Following the tradition from probability theory, random metric spaces and random normed spaces were presented in the course of the development of the theory of probabilistic metric spaces, where the random metric between two points or the random norm of a vector is described by a nonnegative random variable, see \cite[Chapters 9 and 15]{SS8305}. But random normed spaces had not obtained a substantial development up to 1989 since they are often endowed with the $(\varepsilon,\lambda)$--topology, which are not locally convex in general and thus the traditional theory of conjugate spaces universally fails. The first substantial advance came in \cite{Guo89}, where Guo introduced the notion of an almost surely bounded random linear functional on random normed spaces and proved the corresponding Hahn--Banach theorem, which leads to the study of random conjugate spaces. Further, Guo introduced the notions of random normed modules (a special class of random normed spaces) and random inner product modules in \cite{YZG91,Guo92,Guo93} (here, we also mention the work \cite{HLR91} of Haydon, et.al, who independently introduced the notion of random normed modules over the real number field in the name of randomly normed $L^0$--modules, as a tool for the study of ultrapowers of Lebesgue--Bochner function spaces), which leads to a series of deep developments of random conjugate spaces \cite{Guo96a,GY96,Guo96b,GL05,Guo08} (here, we also mention the famous work \cite{HR87} of Hansen and Richard, who independently proved the Riesz representation theorem of random conjugate spaces for a class of special complete random inner product modules--conditional Hilbert spaces, and gave its applications in representing the equilibrium price). As a random generalization of a locally convex space, random locally convex modules were introduced by Guo in \cite{Guo01} and deeply developed in \cite{GC09,GP01,GXC09}. It should be pointed out that random functional analysis was developed under the $(\varepsilon,\lambda)$--topology before 2009.

It is well known that a locally convex space can be defined in two equivalent ways--one by a family of seminorms and the other by a base of convex neighborhoods. Guo \cite{Guo01} gave a random generalization of the first kind by means of a family of $L^0$--seminorms. In 2009, Motivated by financial applications Filipovi\'{c}, Kupper and Vogelpoth \cite{FKV09} gave a random generalization of the second kind by means of a family of $L^0$--convex neighborhoods, which leads to the notion of a locally $L^0$--convex module as well as another kind of topology for a random locally convex module, called the locally $L^0$--convex topology. The central purpose of \cite{FKV09} is an attempt to establish random convex analysis, providing an analytical basis for conditional convex risk measures, but such an attempt is realized by Guo, et.al. in \cite{GZW17,GZZ15a,GZZ15b}. Following \cite{FKV09}, Guo \cite{Guo10} introduced the notion of the countable concatenation property (also called $\sigma$--stability or stability) for a subset of an $L^0(\mathcal{F},K)$--module in order to establish the inherent connections between the two theories derived from the two kinds of topologies (namely the $(\varepsilon,\lambda)$--topology and the locally $L^0$--convex topology) for a random locally convex space. Based on \cite{Guo10}, a basic random convex analysis was established in \cite{GZW17,GZZ15a,GZZ15b} with applications to conditional convex risk measures \cite{GZZ14}. Guo's work \cite{Guo10} also stimulated a series of subsequent researches \cite{CKV15,CVMM17,DKKS13,FM14a,FM14b,GS11,GY12,Wu12,Wu13,WG15,Zap17}, in particular, the notion of the countable concatenation property was frequently employed in \cite{CKV15,DKKS13,GS11} so that a great number of basic results in real analysis and linear algebra can be generalized from Euclidean spaces to random Euclidean spaces.

Hansen and Richard's famous paper \cite{HR87} first studied the optimization problem of conditional variance given conditional mean, which is typically an optimization problem of an $L^0$--convex function defined on a closed $L^0$--convex subset of a complete random inner product module, see Remark \ref{Remark 3.9} for details. With the advent of random convex analysis and its applications to conditional convex risk measures, it is urgent to establish some general principles for random convex optimization and random variational inequalities. Therefore, it is very necessary for us to first investigate compactness on closed $L^0$--convex subsets of random locally convex modules since $L^0$--convex functions defined on closed $L^0$--convex subsets are our objective functions and have played active roles in random convex analysis. Compactness is an important topic in analysis, Guo \cite{Guo08} earlier found that the closed $L^0$--convex subsets frequently occurring in the theory of random normed modules are rarely compact under the $(\varepsilon,\lambda)$--topology, so are under the stronger locally $L^0$--convex topology, namely the conventional theory of compactness does not meet our needs. This forces us to further and deeply investigate Gordan \v{Z}itkovi\'{c}'s idea of convex compactness developed in \cite{Zit10} since a random locally convex module endowed with the $(\varepsilon,\lambda)$--topology is not locally convex in general. We first introduce the notion of $L^0$--convex compactness for a special class of closed convex subsets--closed $L^0$--convex subsets of a Hausdorff topological module over the topological algebra $L^0(\mathcal{F},K)$, then we continue to develop the theory of $L^0$--convex compactness with a series of characterization theorems, in particular we make use of the theory of random conjugate spaces to establish a characterization theorem of James type for a closed $L^0$--convex subset of a complete random normed module to have $L^0$--convex compactness, from which we can derive a surprising fact that the two notions of convex compactness and $L^0$--convex compactness coincide for closed $L^0$--convex subsets of a complete random normed module, Wu and Zhao \cite{WZ19} recently have extended the equivalence to the context of a complete random locally convex module, it should be also pointed out that the equivalence does not reduce the value of the notion of $L^0$--convex compactness, to the contrary the notion leads to a lot of new determination theorems for a closed $L^0$--convex subset to be convexly compact, see Theorem \ref{Theorem 2.21} and Corollary \ref{Corollary 2.23} for details. Further, based on $L^0$--convex compactness, some basic results on random convex optimization can be obtained, see Theorems \ref{Theorem 3.6} and \ref{Theorem 3.8} and Remark \ref{Remark 3.9}. Finally, the basic theorems on classical variational inequalities of Minty type and Brezis type (see \cite{ET99}) are also successfully generalized to the corresponding random settings. Here, we would like to emphasize our new methods and skills in the random settings: since the weak sequence compactness method frequently employed in classical cases \cite{ET99} is no longer valid, we are forced to discover the $L^0$--convex compactness method, for example, we think of Lemma \ref{Lemma 4.4} for the proof of Theorem \ref{Theorem 4.1}; Besides, a random locally convex module can be endowed with the two kinds of topologies, the work of this paper often needs to simultaneously consider them in order to arrive at our aim, which is another space different from the case of classical locally convex spaces. By the way, we also naturally consider applications of the theory of $L^0$--convex compactness to the optimization problem of conditional convex risk measures, we will specially study it in \cite{GZWW19} since conditional convex risk measures are not strictly $L^0$--convex and coercive. The theory of $L^0$--convex compactness can be also used to establish the fixed point theorem for nonexpansive mappings in complete random normed modules \cite{GZWG19}.

The remainder of this paper is organized as follows: Section \ref{Section 2} first recapitulates some known basic notions and facts and then introduces the concept of $L^0$--convex compactness and develop its theory with a series of characterization theorems. Section \ref{Section 3} first proves that a proper, stable and $\mathcal{T}_c$--lower semicontinuous $L^0$--quasiconvex function on an $L^0$--convexly compact set can attain its minimum, and we further establish a Minty type characterization for a minimum point of a G\^{a}teaux--differentiable $L^0$--convex function by variational inequalities. Finally, Section \ref{Section 4} establishes an existence criterion for the solutions of variational inequalities of ``elliptic" type for an $L^0$--convex function defined on a random reflexive random normed module.

Throughout this paper, unless otherwise stated, $(\Omega,\mathcal{F},P)$ always denotes a given probability space, $K$ the scalar field $R$ of real numbers or $C$ of complex numbers, $L^0(\mathcal{F},K)$ the algebra of equivalence classes of $K$--valued $\mathcal{F}$--measurable random variables defined on $(\Omega,\mathcal{F},P)$, in particular we simply write $L^0(\mathcal{F})$ for $L^0(\mathcal{F},R)$ when no confusion occurs. Besides, $\bar{L}^0(\mathcal{F})$ (namely, $\bar{L}^0(\mathcal{F},R)$) stands for the set of equivalence classes of extended real--valued random variables on $(\Omega,\mathcal{F},P)$. Here, equivalence is understood as usual, namely two random variables are equivalent if they equal $P$--almost surely. Proposition \ref{Proposition 1.3} below can be regarded as a random version of the classical supremum principle. The partial order $\leq$ on $\bar{L}^0(\mathcal{F})$ is defined by $\xi \leq \eta$ iff $\xi^0(\omega) \leq \eta^0(\omega)$ for $P$--almost surely all $\omega \in \Omega$, where $\xi^0$ and $\eta^0$ are arbitrarily chosen representatives of $\xi$ and $\eta$ respectively.
\begin{proposition}\label{Proposition 1.3} \cite{DS58}.
 $(\bar{L}^0(\mathcal{F}), \leq)$ is a complete lattice, for any nonempty subset $H$ of $\bar{L}^0(\mathcal{F})$, $\bigvee H$ and $\bigwedge H$ denote the supremum and infimum of $H$, respectively, and the following statements hold:
\begin{enumerate}[(1)]
\item There exists two sequences $\{ a_n, n \in N \}$ and $\{ b_n, n \in N \}$ in $H$ such that $\bigvee_{n \geq 1} a_n = \bigvee H$ and $\bigwedge_{n \geq 1}b_n = \bigwedge H$.
\item If $H$ is directed upwards $($downwards$)$, namely there exists $h_3 \in H$ for any $h_1$ and $h_2 \in H$ such that $h_3 \geq h_1 \bigvee h_2$ $($resp., $h_3 \leq h_1 \bigwedge h_2)$, then $\{ a_n, n \in N \}$ $($resp., $\{ b_n, n \in N \})$ can be chosen as nondecreasing $($resp., nonincreasing$)$.
\item As a sublattice of $\bar{L}^0(\mathcal{F})$, $L^0(\mathcal{F})$ is conditionally complete, namely any nonempty subset with an upper $($resp., a lower$)$ bound has a supremum $($resp., an infimum$)$.
\end{enumerate}
\end{proposition}
In the field of probability theory or mathematical finance, Proposition \ref{Proposition 1.3} often occurs in a different (but equivalent) version: let $\bar{\mathcal{L}}^0(\mathcal{F})$ be the set of extended real--valued random variables on $(\Omega, \mathcal{F},P)$, an essential order $\leq$ on $\bar{\mathcal{L}}^0(\mathcal{F})$ is defined by $\xi \leq \eta$ iff $\xi(\omega) \leq \eta(\omega)$ for $P$--almost surely all $\omega \in \Omega$, then any nonempty subset $H$ of $\bar{\mathcal{L}}^0(\mathcal{F})$ has an essential supremum and an essential infimum, denoted by esssup $H$ and essinf $H$, respectively, it is clear that esssup $H$ and essinf $H$ are unique in the sense of $P$--almost surely equality. Further, for any nonempty subfamily $\mathcal{A}$ of $\mathcal{F}$, esssup $\mathcal{A}$ denotes such an $\mathcal{F}$--measurable set $G$ that $I_G = esssup \{ I_A : A \in \mathcal{A} \}$, called an essential supremum of $\mathcal{A}$, similarly, one can understand essinf $\mathcal{A}$. Here, $I_A$ denotes the characteristic function of $A$, namely $I_A(\omega) = 1$ for $\omega \in A$ and 0 otherwise.

As usual, throughout this paper we denote by $\tilde{I}_A$ the equivalence class of $I_A$ for any $A \in \mathcal{F}$. For two elements $\xi$ and $\eta$ of $\bar{L}^0(\mathcal{F})$, $\xi > \eta$ means that $\xi \geq \eta$ but $\xi \neq \eta$. $L^0_+(\mathcal{F})$ stands for the set $\{ \xi \in L^0(\mathcal{F}) ~|~ \xi \geq 0 \}$ and $\bar{L}^0_+(\mathcal{F}) = \{ \xi \in \bar{L}^0(\mathcal{F})~|~\xi \geq 0 \}$. Finally, for $\xi$ and $\eta$ in $\bar{L}^0(\mathcal{F})$ and $A \in \mathcal{F}$, $\xi > \eta$ on $A$ means that $\xi^0(\omega) > \eta^0(\omega)$ for $P$--almost surely all $\omega \in A$ for arbitrarily chosen representatives $\xi^0$ and $\eta^0$ of $\xi$ and $\eta$, respectively, similarly, one can understand $\xi \geq \eta$ on $A$.

\section{$L^0$--convex compactness and its characterization} \label{Section 2}
The main results of this section are Proposition \ref{Proposition 2.13}, Corollary \ref{Corollary 2.14}, Theorem \ref{Theorem 2.16}, Theorem \ref{Theorem 2.17}, Theorem \ref{Theorem 2.21} and Corollary \ref{Corollary 2.23} below, let us first give some preliminaries before the main results are stated and proved.
\begin{definition}\label{Definition 2.1} \cite{Guo92,Guo93,Guo10}
An ordered pair $(E,\| \cdot \|)$ is called a random normed module $($briefly, an $RN$--module$)$ over the scalar field $K$ with base $(\Omega,\mathcal{F},P)$ if $E$ is a left module over the algebra $L^0(\mathcal{F},K)$ $($briefly, an $L^0(\mathcal{F},K)$--module$)$ and $\| \cdot \|$ is a mapping from $E$ to $L^0_+(\mathcal{F})$ such that the following axioms are satisfied:
\begin{enumerate}[(RNM-1)]
\item $\| \xi x \| = |\xi| \|x\|$ for any $\xi \in L^0(\mathcal{F}, K)$ and any $x \in E$;
\item $\|x+y\| \leq \|x\| + \|y\|$ for all $x$ and $y \in E$;
\item $\|x\| = 0$ implies $x = \theta$ $($the null in $E)$.
\end{enumerate}
In addition, $\| \cdot \|$ is called the $L^0$--norm on $E$ and $\|x\|$ the $L^0$--norm of $x$ for any $x \in E$. A mapping $\| \cdot \| : E \to L^0_+(\mathcal{F})$ is called an $L^0$--seminorm on $E$ if it only satisfies $(RNM-1)$ and $(RNM-2)$ as above.
\end{definition}
\begin{remark}
Similarly, one can understand the notions of a random inner product module $($briefly, an $RIP$--module$)$ and a random locally convex module $($briefly, an $RLC$--module$)$. Especially, an ordered pair $(E,\mathcal{P})$ is called an $RLC$--module over the scalar field $K$ with base $(\Omega,\mathcal{F},P)$ if $E$ is an $L^0(\mathcal{F},K)$--module and $\mathcal{P}$ is a family of $L^0$--seminorms on $E$ such that $\bigvee\{ \|x\| : \| \cdot \| \in \mathcal{P} \} = 0 $ implies $x = \theta$ $($namely $\mathcal{P}$ is separated$)$.
\end{remark}

The most simplest example of $RN$ modules is $L^0(\mathcal{F},K)$ with the $L^0$--norm $\| \cdot \| := |\cdot|$ (namely, the absolute value mapping). When $K$ is replaced by an arbitrary Banach space $B$, one can have a more general $RN$--module $L^0(\mathcal{F},B)$, which was deeply studied in connection with the Lebesgue--Bochner function spaces at the early stage of $RN$ modules \cite{Guo92,YG94,HLR91,Guo96a}.

Example \ref{Example 2.3} below is of fundamental importance for financial applications.
\par
Let $(\Omega,\mathcal{E},P)$ be a probability space and $\mathcal{F}$ a $\sigma$--subalgebra of $\mathcal{E}$, let us first recall from \cite{L78} the notion of a generalized conditional mathematical expectation operator $E[\cdot~|~\mathcal{F}] : \bar{L}^0_+(\mathcal{E}) \to \bar{L}^0_+(\mathcal{F})$ defined by $E[\xi~|~\mathcal{F}] = \lim_{n \to \infty} E[\xi \wedge n~|~\mathcal{F}]$. $\xi \in L^0(\mathcal{E})$ is said to be condotionally integrable with respect to $\mathcal{F}$ if $E[|\xi|~|~\mathcal{F}] < + \infty$ a.s., at this time $E[\xi~|~\mathcal{F}] := E[\xi^+~|~\mathcal{F}] - E[\xi^-~|~\mathcal{F}]$ is called the conditional expectation of $\xi$. Further, let $1 \leq p \leq +\infty$, $|||\xi|||_p := E[|\xi|^p~|~\mathcal{F}]^{1/p}$ for $p < +\infty$ and $|||\xi|||_{\infty} := \bigwedge \{ \eta \in \bar{L}^0_+(\mathcal{F})~|~|\xi| \leq \eta \}$, then it is well known that $|||\xi|||_p < +\infty$ a.s. if and only if $\xi$ can be written as $\xi = \xi_1 \cdot \xi_2$ for some $\xi_1 \in L^0(\mathcal{F})$ and $\xi_2 \in L^p(\mathcal{E})$, where $L^p(\mathcal{E})$ stands for the usual Banach space of $p$--integrable $(p < +\infty)$ or essentially bounded $(p=+\infty)$ functions on $(\Omega,\mathcal{E},P)$.
\begin{example} \label{Example 2.3} \cite{HR87,FKV09}
Let $(\Omega,\mathcal{E},P)$ be a probability space and $\mathcal{F}$ a $\sigma$--subalgebra of $\mathcal{E}$. For any fixed extended positive real number $p \in [1,+\infty]$, let $L^p_{\mathcal{F}}(\mathcal{E}) = L^0(\mathcal{F}) \cdot L^p(\mathcal{E})
:=\{ \xi x ~:~ \xi \in L^0(\mathcal{F})$ and  $x \in L^p(\mathcal{E}) \}$, $|||\cdot |||_p: L^p_{\mathcal{F}}(\mathcal{E}) \to L^0_{+}(\mathcal{F})$ is defined as follows:
$$|||y|||_p=\left\{
            \begin{array}{ll}
              E[|y|^p ~|~ \mathcal{F}]^{\frac{1}{p}}, & \hbox{when $1\leq p<+\infty$;} \\

              \bigwedge \{\eta \in \L^0_+(\Omega,\mathcal{F},P)~|~|y|\leq \eta\}, &  \hbox{when $p=+\infty$.}
            \end{array}
          \right.$$
Then $(L^p_{\mathcal{F}}(\mathcal{E}),||| \cdot |||_p)$ is an $RN$--module over $R$ with base $(\Omega,\mathcal{F},P)$, specially $L^2_{\mathcal{F}}(\mathcal{E})$ is an $RIP$--module. Historically, $L^2_{\mathcal{F}}(\mathcal{E})$ first occurred in \cite{HR87} and general $L^p_{\mathcal{F}}(\mathcal{E})$ was given in \cite{FKV09}.
\end{example}

The $(\varepsilon,\lambda)$--topology for $L^0(\mathcal{F},K)$ was introduced by B.Schweizer and A.Sklar in \cite{SS8305}: for any given positive numbers $\varepsilon$ and $\lambda$ such that $\lambda < 1$, let $U_{\theta}(\varepsilon,\lambda) = \{ \xi \in L^0(\mathcal{F},K)~|~P\{ \omega \in \Omega ~|~ |\xi(\omega)| < \varepsilon \} > 1-\lambda \}$. Then $\{ U_{\theta}(\varepsilon,\lambda)~|~ \varepsilon >0,~ 0 < \lambda <1 \}$ forms a local base at $\theta$ of some metrizable linear topology for $L^0(\mathcal{F},K)$, which is called the $(\varepsilon,\lambda)$--topology for $L^0(\mathcal{F},K)$ and is exactly the topology of convergence in probability. It is obvious that $L^0(\mathcal{F},K)$ is also a topological algebra over $K$ under the $(\varepsilon,\lambda)$--topology. In fact, B.Schweizer and A.Sklar also introduced the $(\varepsilon,\lambda)$--topology for random normed spaces which are more general than random normed modules, see \cite{SS8305}.

To introduce the $(\varepsilon,\lambda)$--topology for a random locally convex module, let $(E,\mathcal{P})$ be a random locally convex module with base $(\Omega,\mathcal{F},P)$, for any finite nonempty subfamily $Q$ of $\mathcal{P}$, $\| \cdot \|_Q ~:~ E \to L^0_+(\mathcal{F})$ defined by $\|x\|_Q = \bigvee\{ \|x\| : \| \cdot \| \in Q \}$ for any $x \in E$ is still an $L^0$--seminorm on $E$. Furthermore, let $U_{\theta}(Q,\varepsilon,\lambda) = \{ x \in E ~|~ P\{ \omega \in \Omega ~|~ \|x\|_Q(\omega) < \varepsilon \} > 1-\lambda \}$ for any finite nonempty subfamily $Q$ of $\mathcal{P}$, $\varepsilon >0$ and $0 < \lambda <1$. Then we have the following:
\begin{proposition}\label{Proposition 2.4} \cite{Guo01,GP01}
Let $(E,\mathcal{P})$ be a random locally convex module over $K$ with base $(\Omega,\mathcal{F},P)$. Then $\{ U_{\theta}(Q,\varepsilon,\lambda) ~|~ Q $ is a finite nonempty subfamily of $\mathcal{P}$, $\varepsilon >0, 0 < \lambda <1 \}$ forms a local base at $\theta$ of some Hausdorff linear topology for $E$, called the $(\varepsilon,\lambda)$--topology. Furthermore, $E$ is a topological module over the topological algebra $L^0(\mathcal{F},K)$ when $E$ and $L^0(\mathcal{F},K)$ are endowed with their respective $(\varepsilon,\lambda)$--topology.
\end{proposition}

In the sequel, the $(\varepsilon,\lambda)$--topology for any random locally convex module is always denoted by $\mathcal{T}_{\varepsilon,\lambda}$. For any random locally convex module $(E,\mathcal{P})$ over $K$ with base $(\Omega,\mathcal{F},P)$, the $L^0(\mathcal{F},K)$--module of continuous module homomorphisms from $(E,\mathcal{T}_{\varepsilon,\lambda})$ to $(L^0(\mathcal{F},K),\mathcal{T}_{\varepsilon,\lambda})$ is called the random conjugate space of $(E,\mathcal{P})$ with respect to the $(\varepsilon,\lambda)$--topology, denoted by $(E,\mathcal{P})^{\ast}_{\varepsilon,\lambda}$ or briefly by $E^{\ast}_{\varepsilon,\lambda}$.
\begin{definition}\label{Definition 2.5} \cite{Guo01,GXC09}
Let $E$ be a left module over the algebra $L^0(\mathcal{F},K)$ $($briefly,an $L^0(\mathcal{F},K)$--module$)$ and $D$ a subset of $E$. $D$ is $L^0$--convex if $\xi x + (1-\xi)y \in D$ for all $x$ and $y \in D$ and $\xi \in L^0_+(\mathcal{F})$ such that $0 \leq \xi \leq 1$. A subset $H$ of $E$ is $L^0$--absorbed by $D$ if there exists $\eta \in L^0_{++}(\mathcal{F})$ such that $\xi H := \{ \xi h ~|~ h \in H \} \subset D$ for all $\xi \in L^0(\mathcal{F},K)$ with $|\xi| \leq \eta$; if every point in $E$ is $L^0$--absorbed by $D$, then $D$ is $L^0$--absorbent. $D$ is $L^0$--balanced if $\xi D \subset D$ for all $\xi \in L^0(\mathcal{F},K)$ such that $|\xi| \leq 1$.
\end{definition}

Let $V_{\theta}(\varepsilon) = \{ \xi \in L^0(\mathcal{F},K) ~|~ |\xi| \leq \varepsilon \}$ for any $\varepsilon \in L^0_{++}(\mathcal{F})$.
In 2009, Filipovi\'{c}, et.al \cite{FKV09} first introduced another kind of topology for $L^0(\mathcal{F},K)$, called the locally $L^0$--convex topology, denoted by $\mathcal{T}_c$, a subset $G$ of $L^0(\mathcal{F},K)$ is open under this topology if there is some $V_{\theta}(\varepsilon)$ for any fixed element $g \in G$ such that $g + V_{\theta}(\varepsilon) \subset G$. It is easy to verify that $(L^0(\mathcal{F},K),\mathcal{T}_c)$ is a topological ring.
In 2009, on the basis of this, Filipovi\'{c},et.al \cite{FKV09} introduced the notion of a locally $L^0$--convex module as follows: an ordered pair $(E,\mathcal{T})$ is a locally $L^0$--convex module if $E$ is an $L^0(\mathcal{F},K)$--module and $\mathcal{T}$ is a topology on $E$ such that $(E,\mathcal{T})$ is a topological module over the topological ring $L^0(\mathcal{F},K)$ and $\mathcal{T}$ has a local base at $\theta$ (the null element of $E$) whose every member is $L^0$--convex, $L^0$--absorbent and $L^0$--balanced, at which time, $\mathcal{T}$ is a locally $L^0$--convex topology on $E$. This leads directly to the following:
\begin{proposition}\label{Proposition 2.6} \cite{FKV09}
Let $(E,\mathcal{P})$ be a random locally convex module over $K$ with base $(\Omega,\mathcal{F},P)$. Then $\{ U_{\theta}(Q,\varepsilon) ~|~ Q$ is a finite nonempty subfamily of $\mathcal{P}$ and $\varepsilon \in L^0_{++}(\mathcal{F}) \}$ forms a local base at $\theta$ of some Hausdorff locally $L^0$--convex topology for $E$, called the locally $L^0$--convex topology induced by $\mathcal{P}$, where $U_{\theta}(Q,\varepsilon) = \{ x \in E ~|~ \|x\|_Q \leq \varepsilon \}$.
\end{proposition}

For the sake of convenience, from now on, the locally $L^0$--convex topology for an arbitrary random locally convex module, its locally $L^0$--convex topology is always denoted by $\mathcal{T}_c$. Furthermore, for any random locally convex module $(E,\mathcal{P})$ over $K$ with base $(\Omega,\mathcal{F},P)$, the $L^0(\mathcal{F},K)$--module of continuous module homomorphisms from $(E,\mathcal{T}_c)$ to $(L^0(\mathcal{F},K),\mathcal{T}_c)$ is called the random conjugate space of $(E,\mathcal{P})$ with respect to the locally $L^0$--convex topology, denoted by $(E,\mathcal{P})^{\ast}_{c}$ or briefly by $E^{\ast}_c$.

In \cite{FKV09}, a family $\mathcal{P}$ of $L^0$--seminorms on an $L^0(\mathcal{F},K)$--module $E$ is said to have the countable concatenation property if the $L^0$--seminorm $\sum^{\infty}_{n=1} \tilde{I}_{A_n} \|\cdot\|_{Q_n}$ still belongs to $\mathcal{P}$ for any sequence $\{ Q_n ~|~ n \in N \}$ of finite nonempty subfamilies of $\mathcal{P}$ and for any countable partition $\{ A_n ~|~ n \in N \}$ of $\Omega$ to $\mathcal{F}$.

Another crucial notion is the following:
\begin{definition}\label{Definition 2.7} \cite{Guo10}
Let $E$ be an $L^0(\mathcal{F},K)$--module and $G$ a subset of $E$. $G$ is said to have the countable concatenation property if there is $g \in G$ for any sequence $\{ g_n ~|~ n \in N \}$ in $G$ and for any countable partition $\{ A_n ~|~ n \in N \}$ of $\Omega$ to $\mathcal{F}$ such that $\tilde{I}_{A_n} g = \tilde{I}_{A_n} g_n$ for all $n \in N$. Furthermore, if $E$ has the countable concatenation property, we always write $H_{cc}(G)$ for the smallest set which contains $G$ and has the countable concatenation property, called the countable concatenation hull of $G$, where $G$ is a subset of $E$.
\end{definition}

In general, $g$ as in Definition \ref{Definition 2.7}, which satisfies $\tilde{I}_{A_n} g = \tilde{I}_{A_n} g_n, \forall n \in N$ for any given $\{ g_n \}$ and $\{ A_n \}$, is unique, for example, this is true for any random locally convex module, at which time we can write $g = \sum^{\infty}_{n=1}\tilde{I}_{A_n}g_n$. It is also easy to verify that the random conjugate space $E^{\ast}_{\varepsilon,\lambda}$ of a random locally convex module $(E,\mathcal{P})$ always has the countable concatenation property. Besides, it is well known that  $L^{p}_{\mathcal{F}}(\mathcal{E})$ has the countable concatenation property for each $p \in [1, +\infty]$, see \cite{GZZ14}.

Proposition \ref{Proposition 2.8} below throughly describes the relation between $E^{\ast}_{\varepsilon,\lambda}$ and $E^{\ast}_c$.
\begin{proposition}\label{Proposition 2.8} \cite{Guo10,GZZ15a}
Let $(E,\mathcal{P})$ be any random locally convex module. Then the following statements hold:
\begin{enumerate}[(1)]
\item $($see \cite{Guo10}$)$. $E^{\ast}_{\varepsilon,\lambda} = E^{\ast}_c$ if $\mathcal{P}$ has the countable concatenation property, specially $E^{\ast}_{\varepsilon,\lambda} = E^{\ast}_c$ for any random normed module $(E, \|\cdot\|)$.
\item $($see \cite{GZZ15a}$)$. $E^{\ast}_{\varepsilon,\lambda} = H_{cc}(E^{\ast}_c)$
\end{enumerate}
\end{proposition}

As (1) of Proposition \ref{Proposition 2.8} shows that $E^{\ast}_{\varepsilon,\lambda} = E^{\ast}_c$ for any random normed module $(E, \|\cdot\|)$ over $K$ with base $(\Omega, \mathcal{F},P)$, so we can simply write $E^{\ast}$ for $E^{\ast}_{\varepsilon,\lambda}$ or $E^{\ast}_c$. At the early stage of random conjugate spaces, it was shown in \cite{Guo92,Guo96b} that a linear operator $f$ from $E$ to $L^0(\mathcal{F},K)$ belongs to $E^{\ast}$ iff $f$ is almost surely bounded, (namely, there exists some $\xi \in L^0_+(\mathcal{F})$ such that $|f(x)| \leq \xi \|x\|$ for all $x \in E$), and an $L^0$--norm $\|\cdot\|^{\ast} : E^{\ast} \to L^0_+(\mathcal{F})$ can be introduced by $\|f\|^{\ast} = \bigwedge \{ \xi \in L^0_+(\mathcal{F}) ~|~ |f(x)| \leq \xi \|x\|$ for all $x \in E \}$ so that $(E^{\ast},\|\cdot\|^{\ast})$ is a random normed module over $K$ with base $(\Omega,\mathcal{F},P)$, likewise, $(E^{\ast \ast},\|\cdot\|^{\ast \ast})$ can be defined, called the random biconjugate space of $E$. In the sequel, we still briefly write $\|\cdot\|$ for $\|\cdot\|^{\ast}$ or $\|\cdot\|^{\ast \ast}$, which will not cause any confusion.

Just as in classical functional analysis, the canonical embedding mapping $J ~:~ (E,\|\cdot\|) \to (E^{\ast \ast}, \|\cdot\|)$ defined by $J(x)(f) = f(x)$ for all $f \in E^{\ast}$ and all $x \in E$, is $L^0$--norm--preserving. Naturally, if $J$ is surjective, then $E$ is said to be random reflexive. Since $E^{\ast}$ is always $\mathcal{T}_{\varepsilon,\lambda}$--complete for any random normed module $(E,\|\cdot\|)$, of course, $E^{\ast \ast}$ is also $\mathcal{T}_{\varepsilon,\lambda}$--complete, so that any random reflexive random normed module is always $\mathcal{T}_{\varepsilon,\lambda}$--complete. Besides, all $L^{p}_{\mathcal{F}}(\mathcal{E})$ are random reflexive for $p \in (1, + \infty)$, see \cite{Guo10}.

Now, we can return to the theme of this section by beginning with the following:
\begin{definition}\label{Definition 2.9}
Let $(E,\mathcal{T})$ be a topological module over the topological algebra $(L^0(\mathcal{F},K),\mathcal{T}_{\varepsilon,\lambda})$ and $G$ a closed $L^0$--convex subset of $E$. $G$ is $L^0$--convexly compact $($or, is said to have $L^0$--convex compactness$)$ if any family of closed $L^0$--convex subsets of $G$ has a nonempty intersection whenever this family has the finite intersection property.
\end{definition}

\begin{remark} \label{Remark 2.10}
When $\mathcal{F}$ reduces to a trivial $\sigma$--algebra, namely $\mathcal{F} = \{ \Omega, \emptyset \}$, $(E,\mathcal{T})$ reduces to an ordinary topological linear space and $G$ to an ordinary convex set, then the concept of $L^0$--convex compactness in Definition \ref{Definition 2.9} reduces to that of convex compactness, which was introduced by G.\v{Z}itkovi\'{c} in \cite{Zit10}, so the $L^0$--convex compactness is a generalization of convex compactness. On the other hand, when $\mathcal{F}$ is a generic $\sigma$--algebra, since a closed $L^0$--convex subset is also a closed convex subset, then it is natural that we should compare the notions of convex compactness and $L^0$--convex compactness. The concept of $L^0$--convex compactness only impose certain restriction on the family of ``closed $L^0$--convex subsets" of $G$, rather than on the larger family of ``closed convex subsets" of $G$, and thus the concept of $L^0$--convex compactness always seems weaker than that of convex compactness for a closed $L^0$--convex subset, but our Theorem \ref{Theorem 2.21} and Remark \ref{Remark 2.22} below show that the two notions coincide for a class of important closed convex subsets--closed $L^0$--convex subsets of a complete $RN$ module, even Wu and Zhao \cite{WZ19} recently have extended the equivalence to the context of a complete random locally convex module, we still would like to retain the notion of $L^0$--convex compactness in order to make this paper self--contained, it is more important that the study of $L^0$--convex compactness can lead to many new determination theorems for a closed $L^0$--convex subset to be convexly compact, see Theorem \ref{Theorem 2.21} and Corollary \ref{Corollary 2.23} below.
\end{remark}

For a subset $G$ of an $L^0(\mathcal{F},K)$--module $E$, $conv_{L^0}(G)$ ($conv(G)$) always denotes the smallest $L^0$--convex (correspondingly, convex) subset containing $G$, called the $L^0$--convex (correspondingly, convex) hull of $G$. Clearly,  $conv_{L^0}(G) \supset conv(G)$.

Similarly to Definition 2.3 of \cite{Zit10}, we have the following:
\begin{definition}\label{Definition 2.11}
Let $\{ x_{\alpha}, \alpha \in \Gamma \}$ be a net in an $L^0(\mathcal{F},K)$--module $E$ and $Fin(\Gamma)$ denote the family of finite subsets of $\Gamma$. A net $\{ y_{\beta}, \beta \in B \}$ is said to be a subnet of $L^0$--convex combinations of $\{ x_{\alpha}, \alpha \in \Gamma \}$ if there exists a mapping $D ~:~ B \to Fin(\Gamma)$ such that the following two items hold:
\begin{enumerate}[(1)]
\item $y_{\beta} \in conv_{L^0} \{ x_{\alpha}, \alpha \in D(\beta) \}$ for each $\beta \in B$;
\item for each $\alpha \in \Gamma$ there exists $\beta \in B$ such that $\alpha^{\prime} \geq \alpha$ for each $\alpha^{\prime} \in \bigcup_{\beta^{\prime} \geq \beta}D(\beta^{\prime})$
\end{enumerate}
\end{definition}

Proposition \ref{Proposition 2.12} below is an $L^0(\mathcal{F})$--version of Proposition 2.4 of \cite{Zit10}, but its proof is omitted since the proof is a word--by--word copy of that of Proposition 2.4 of \cite{Zit10}.
\begin{proposition}\label{Proposition 2.12}
A closed $L^0$--convex subset $G$ of a topological module $(E,\mathcal{T})$ over the topological algebra $(L^0(\mathcal{F},K),\mathcal{T}_{\varepsilon,\lambda})$ is $L^0$--convexly compact iff for any net $\{ x_{\alpha}, \alpha \in \Gamma \}$ in $G$ there exists a subnet $\{ y_{\beta}, \beta \in B \}$ of $L^0$--convex combinations of $\{ x_{\alpha}, \alpha \in \Gamma \}$ such that $\{ y_{\beta}, \beta \in B \}$ converges to some $y \in G$.
\end{proposition}

As was stated in \cite{Zit10}, Proposition \ref{Proposition 2.12} is a characterization in terms of generalized sequences, we would like to give some variants of Proposition \ref{Proposition 2.12}, which provide much convenience for the purpose of this paper. In particular, these variants give the more precise relation between $B$ and $\Gamma$ in Proposition \ref{Proposition 2.12}, as a consequence, Corollary \ref{Corollary 2.14} and Theorem \ref{Theorem 2.16} below will play a crucial role in the sequel of this paper.
\begin{proposition}\label{Proposition 2.13}
Let $(E,\mathcal{T})$ be a topological module over the topological algebra $(L^0(\mathcal{F},K),\mathcal{T}_{\varepsilon,\lambda})$ and $\mathcal{U}$ a local base of the topology $\mathcal{T}$. Then a closed $L^0$--convex subset $G$ of $E$ is $L^0$--convexly compact iff for any net $\{ x_{\alpha}, \alpha \in \Gamma \}$ in $G$ there exists a subnet $\{ y_{(\alpha,U)}, (\alpha,U) \in \Gamma \times \mathcal{U} \}$ of $L^0$--convex combinations of $\{ x_{\alpha}, \alpha \in \Gamma \}$ such that $y_{(\alpha,U)}$ converges to some $y \in G$ and each $y_{(\alpha,U)} \in (y+U) \bigcap conv_{L^0}\{ x_{\alpha^{\prime}} ~|~ \alpha^{\prime} \geq \alpha \}$, where $\Gamma \times \mathcal{U}$ is directed by $(\alpha_{1},U_1) \leq (\alpha_{2},U_2)$ iff $\alpha_{1} \leq \alpha_{2}$ and $U_2 \subset U_1$.
\end{proposition}
\begin{proof}
Sufficiency is similar to that of Proposition 2.4 of \cite{Zit10}, so is omitted.

Necessity. By the necessity of Proposition \ref{Proposition 2.12} there exists a subnet $\{ z_{\beta}, \beta \in B \}$ of $L^0$--convex combinations of $\{ x_{\alpha}, \alpha \in \Gamma \}$ such that $z_{\beta}$ converges to some $y \in E$. By Definition \ref{Definition 2.11}, there exists a mapping $D ~:~ B \to Fin(\Gamma)$ such that the following two items hold:
\begin{enumerate}[(1)]
\item $z_{\beta} \in conv_{L^0} \{ x_{\alpha}, \alpha \in D(\beta) \}$ for each $\beta \in B$;
\item for each $\alpha \in \Gamma$ there exists $\beta \in B$ such that $\alpha^{\prime} \geq \alpha$ for each $\alpha^{\prime} \in \bigcup_{\beta^{\prime} \geq \beta}D(\beta^{\prime})$.
\end{enumerate}

From (2), one can see that for each $\alpha \in \Gamma$ there exists $\beta \in B$ such that $z_{\beta^{\prime}} \in conv_{L^0} \{ x_{\alpha^{\prime}}, \alpha^{\prime} \in D(\beta^{\prime}) \} \subset conv_{L^0}\{ x_{\alpha^{\prime \prime}}, \alpha^{\prime \prime} \geq \alpha \}$ for each $\beta^{\prime} \geq \beta$, so that $y$ must belong to $\overline{conv_{L^0}\{ x_{\alpha^{\prime}}, \alpha^{\prime} \geq \alpha \}}$ (namely, the closure of $conv_{L^0}\{ x_{\alpha^{\prime}}, \alpha^{\prime} \geq \alpha \}$). Of course, for each $\alpha \in \Gamma$ and each $U \in \mathcal{U}$ there exists $y_{(\alpha,U)} \in (y + U) \bigcap conv_{L^0}\{ x_{\alpha^{\prime}}, \alpha^{\prime} \geq \alpha \}$, it is also obvious that the net $\{ y_{(\alpha,U)}, (\alpha,U) \in \Gamma \times \mathcal{U} \}$ converges to $y$.
\end{proof}
\begin{corollary} \label{Corollary 2.14}
Let $(E,\mathcal{T})$ be a metrizable topological module over the topological algebra $(L^0(\mathcal{F},K),\mathcal{T}_{\varepsilon,\lambda})$ and $G$ an $L^0$--convexly compact closed $L^0$--convex subset of $E$. Then for each sequence $\{ x_n, n \in N \}$ in $G$ there exists a sequence $\{ y_n, n \in N \}$ of forward $L^0$--convex combinations of $\{ x_n, n \in N \}$ $($namely $y_n \in conv_{L^0} \{ x_k, k \geq n \}$ for each $n \in N)$ such that $y_n$ converges to some $y \in G$.
\end{corollary}
\begin{proof}
Let $\mathcal{U} = \{ U_n, n \in N \}$ be a countable local base of $E$ such that $U_{n+1} \subset U_n$ for each $n \in N$. Further, let $y$ be as in the proof of Proposition \ref{Proposition 2.13}, then $y \in \overline{conv_{L^0}\{ x_k, k \geq n \}}$ for each $n \in N$, taking $y_n \in (y + U_n) \bigcap conv_{L^0}\{ x_k, k \geq n \}$ for each $n \in N$ completes the proof.
\end{proof}

For a random locally convex module $(E,\mathcal{P})$ over $K$ with base $(\Omega,\mathcal{F},P)$. $(E, \mathcal{T}_{\varepsilon,\lambda})$ is just a topological module over the topological algebra $(L^0(\mathcal{F},K),\mathcal{T}_{\varepsilon,\lambda})$. For an $L^0$--convexly compact closed $L^0$--convex subset $G$ of $E$ and any net $\{ x_{\alpha}, \alpha \in \Gamma \}$ in $G$, the subnet $\{ y_{\beta}, \beta \in B \}$ or $\{ y_{(\alpha,U)}, (\alpha,U) \in \Gamma \times \mathcal{U} \}$ of $L^0$--convex combinations of $\{ x_{\alpha}, \alpha \in \Gamma \}$, which are obtained as in Proposition \ref{Proposition 2.12} and Proposition \ref{Proposition 2.13} respectively, converges in $\mathcal{T}_{\varepsilon,\lambda}$ to $y$. However, we are often forced to look for a net convergent to this $y$ with respect to the stronger topology--the locally $L^0$--convex topology $\mathcal{T}_c$. For this, let us recall some known results: \cite[Theorem 3.12]{Guo10} says that $\overline{M}_{\varepsilon,\lambda} = \overline{M}_{c}$ for any subset $M$ of $E$ such that $M$ has the countable concatenation property, where $\overline{M}_{\varepsilon,\lambda}$ and $\overline{M}_{c}$ stand for the closures of $M$ under the $(\varepsilon,\lambda)$--topology $\mathcal{T}_{\varepsilon,\lambda}$ and the locally $L^0$--convex topology $\mathcal{T}_c$ respectively. \cite[Theorem 3.12]{Guo10} leads directly to the following:
\begin{lemma} \label{Lemma 2.15}\cite[Lemma 3.10]{GZZ15b}
Let $(E,\mathcal{P})$ be a random locally convex module over $K$ with base $(\Omega,\mathcal{F},P)$ such that $E$ has the countable concatenation property. If $M$ is $L^0$--convex, then $\overline{M}_{\varepsilon,\lambda} = [H_{cc}(M)]^{-}_{c}$, where $H_{cc}(M)$ is the countable concatenation hull of $M$.
\end{lemma}
\begin{theorem}\label{Theorem 2.16}
Let $(E,\mathcal{P})$ be a random locally convex module over $K$ with base $(\Omega,\mathcal{F},P)$ such that $E$ has the countable concatenation property. Further, if $G$ is an $L^0$--convexly compact $\mathcal{T}_{\varepsilon,\lambda}$--closed $L^0$--convex subset of the topological module $(E,\mathcal{T}_{\varepsilon,\lambda})$ and $\{ x_{\alpha}, \alpha \in \Gamma \}$ is a net in $G$, then there exists a net $\{ y_{(\alpha ,U)}, (\alpha,U) \in \Gamma \times \mathcal{U} \}$ convergent to some $y \in G$ with respect to the locally $L^0$--convex topology $\mathcal{T}_c$, where $\mathcal{U}$ is an arbitrarily chosen local base of $\mathcal{T}_c$, for example $\mathcal{U} = \{ U_{\theta}(Q,\varepsilon) ~|~ Q \subset \mathcal{P}$ finite, $\varepsilon \in L^0_{++}(\mathcal{F}) \}$ as in Proposition \ref{Proposition 2.6}, as usual, $\Gamma \times \mathcal{U}$ is directed by $(\alpha_1, U_1) \leq (\alpha_2, U_2)$ iff $\alpha_1 \leq \alpha_2$ and $U_2 \subset U_1$. Besides, $y_{(\alpha,U)} \in (y + U) \bigcap H_{cc}(conv_{L^0}\{ x_{\alpha^{\prime}}, \alpha^{\prime} \geq \alpha \})$ for each $(\alpha, U) \in \Gamma \times \mathcal{U}$.
\end{theorem}
\begin{proof}
Let $y$ be as stated in Proposition \ref{Proposition 2.13}, then $y \in \overline{(conv_{L^0}\{ x_{\alpha^{\prime}}, \alpha^{\prime} \geq \alpha \})}_{\varepsilon,\lambda}$ for each $\alpha \in \Gamma$. Now, by Lemma \ref{Lemma 2.15}, $y \in \overline{H_{cc}(conv_{L^0}\{ x_{\alpha^{\prime}}, \alpha^{\prime} \geq \alpha \})}_{c}$, then for each $(\alpha, U) \in \Gamma \times \mathcal{U}$, taking $y_{(\alpha,U)} \in (y + U) \bigcap H_{cc}(conv_{L^0}\{ x_{\alpha^{\prime}}, \alpha^{\prime} \geq \alpha \})$ completes the proof.
\end{proof}

The characterization concerning $L^0$--convexly compact sets in $L^0(\mathcal{F})$ is the following:
\begin{theorem}\label{Theorem 2.17}
A $\mathcal{T}_{\varepsilon,\lambda}$--closed $L^0$--convex subset $G$ of $L^0(\mathcal{F})$ is $L^0$--convexly compact iff $G$ is bounded in order, namely there exists $\xi \in L^0_+(\mathcal{F})$ such that $|\eta| \leq \xi$ for any $\eta \in G$.
\end{theorem}
\begin{proof}
(1). Necessity. Let $a = \bigwedge G$ and $b = \bigvee G$, we only need to assert that both $a$ and $b$ belong to $L^0(\mathcal{F})$, and only give the proof of $b \in L^0(\mathcal{F})$, \ since the other is similar. By contradiction method: suppose $b ~\overline{\in} ~L^0(\mathcal{F})$, then $A := (b = +\infty)$ must satisfy $P(A) > 0$. Since the $L^0$--convexity of $G$ implies that $G$ is directed upwards, by Proposition \ref{Proposition 1.3} there exists a nondecreasing sequence $\{ g_n ~|~ n \in N \}$ in $G$ such that $g_n$ converges almost surely to $b$, and hence $\lim_{n \to \infty} \inf_{k \geq n} g_k = \underline{\lim}_{n} g_n = b$. According to $L^0$--convex compactness of $G$ and Corollary \ref{Corollary 2.14}, we can get a sequence $\{ g^{\prime}_{n} ~|~ n \in N \}$ in $G$ such that $g^{\prime}_n \in conv_{L^0} \{ g_k ~|~ k \geq n \}$ for each $n \in N$ and $y \in L^0(\mathcal{F})$ satisfying that $\{ g^{\prime}_n ~|~ n \in N \}$ converges in probability to $y$. It is obvious that $g^{\prime}_n \geq \inf_{k \geq n} g_{k}$ for each $n \in N$, which implies $y \geq b$, in particular $y \geq b = +\infty$ on $A$, a contradiction to $y \in L^0(\mathcal{F})$.

(2). Sufficiency.  Since $G$ must be a random closed interval of $L^0(\mathcal{F})$, let $G = [a,b]$ with $a$ and $b \in L^0(\mathcal{F})$ and $a \leq b$. Given any family $\{ G_{\alpha}, \alpha \in I \}$ of $\mathcal{T}_{\varepsilon,\lambda}$--closed $L^0$--convex subsets of $G$ such that this family has the finite intersection property, let $Fin(I)$ be the family of finite nonempty subsets of $I$ and $G_F = \bigcap_{\alpha \in F}G_{\alpha}$ for each $F \in Fin(I)$, then each $G_F$ is a nonempty random closed interval of $L^0(\mathcal{F})$, written as $[a_F,b_F]$, where $a_F$ and $b_F \in L^0(\mathcal{F})$ and $a_F \leq b_F$. Since $Fin(I)$ is a directed set via the order $\leq : F_1 \leq F_2$ iff $F_1 \subset F_2$, then $\{ a_F ~|~ F \in Fin(I) \}$ is a nondecreasing net, while $\{ b_F, F \in Fin(I) \}$ is a nonincreasing net. Putting $\xi_{1} = \bigvee \{ a_F ~|~ F \in Fin(I) \}$ and $\xi_{2} = \bigwedge \{ b_F ~|~ F \in Fin(I) \}$, then $\xi_{1}$ and $\xi_{2}$ both belong to $[a,b]$ and $\xi_{1} \leq \xi_{2}$, so that $[\xi_{1}, \xi_{2}] \subset \bigcap_{\alpha \in I} G_{\alpha}$.
\end{proof}

Let $(E, \mathcal{P})$ be a random locally convex module over $K$ with base $(\Omega,\mathcal{F},P)$. Since the $(\varepsilon,\lambda)$--topology $\mathcal{T}_{\varepsilon,\lambda}$ on $E$ is a linear topology, one can speak of $\mathcal{T}_{\varepsilon,\lambda}$--boundedness. It is known and also clear that a subset $G$ of $E$ is $\mathcal{T}_{\varepsilon,\lambda}$--bounded iff for each $\|\cdot\| \in \mathcal{P}$ one can have $\lim_{n \to +\infty} \sup_{g \in G} P(\|g\| \geq n) = 0$, namely for each $\|\cdot\| \in \mathcal{P}$, $\{ \|g\| ~|~ g \in G \} \subset L^0_+(\mathcal{F}) \subset L^0(\mathcal{F})$ is bounded in probability (or, probabilistically bounded in terms of \cite{SS8305}). Another notion of boundedness is crucial for this paper: a subset $G$ of $E$ is said to be almost surely bounded if $\bigvee \{ \|g\| ~|~ g \in G \} \in L^0_+(\mathcal{F})$ for each $\|\cdot\| \in \mathcal{P}$, namely $\{ \|g\| ~|~ g \in G \} \subset L^0_+(\mathcal{F}) \subset L^0(\mathcal{F})$ is bounded in order for each $\|\cdot\| \in \mathcal{P}$. According to the resonance theorem in random normed modules \cite{Guo92,Guo96b,Guo13}, the following  result was already obtained and will be used in this paper:
\begin{proposition}\label{Proposition 2.18} \cite{Guo92,Guo96b}
Let $(E, \mathcal{P})$ be a random locally convex module over $K$ with base $(\Omega,\mathcal{F},P)$ and $G$ a subset of $E$. Then the following statements are true:
 \begin{enumerate}[(1)]
\item $G$ is $\mathcal{T}_{\varepsilon,\lambda}$--bounded iff $f(G)$ is $\mathcal{T}_{\varepsilon,\lambda}$--bounded in $L^0(\mathcal{F},K)$ for each $f \in E^{\ast}_{\varepsilon,\lambda}$.
\item $G$ is almost surely bounded iff $f(G)$ is almost surely bounded in $L^0(\mathcal{F},K)$ for each $f \in E^{\ast}_{\varepsilon,\lambda}$.
\end{enumerate}
\end{proposition}
\begin{lemma}\label{Lemma 2.19}
Let $(E, \mathcal{P})$ be a random locally convex module over $K$ with base $(\Omega,\mathcal{F},P)$ and $G$ is a $\mathcal{T}_{\varepsilon,\lambda}$--closed $L^0$--convex subset of $E$. If $G$ is $L^0$--convexly compact with respect to $\mathcal{T}_{\varepsilon,\lambda}$, then $G$ must be almost surely bounded.
\end{lemma}
\begin{proof}
First, we prove that $f(G)$ is $\mathcal{T}_{\varepsilon,\lambda}$--closed for each $f \in E^{\ast}_{\varepsilon,\lambda}$. For this, let $\xi$ belong to $\overline{f(G)}$ (the $\mathcal{T}_{\varepsilon,\lambda}$--closure of $f(G)$), then there exists a sequence $\{ x_n, n \in N \}$ in $G$ such that $\{ f(x_n), n \in N \}$ converges in $\mathcal{T}_{\varepsilon,\lambda}$ to $\xi$, namely converges in probability to $\xi$. We can, without loss of generality, assume that $\{ f(x_n), n \in N \}$ converges almost surely to $\xi$. By Proposition \ref{Proposition 2.12} there exists a subnet $\{ y_\beta, \beta \in B \}$ of $L^0$--convex combinations of $\{ x_n, n \in N \}$ such that $\{ y_\beta, \beta \in B \}$ converges to some $y \in G$. In particular, by definition, $\{ y_\beta, \beta \in B \}$ satisfies the following two items with $D: B \to Fin(N)$ as in Definition \ref{Definition 2.11}:
 \begin{enumerate}[(1)]
\item $y_\beta \in conv_{L^0} \{ x_n, n \in D(\beta) \}$ for each $\beta \in B$;
\item For each $n \in N$ there exists $\beta_n \in B$ such that $m \geq n$ for each $m \in \bigcup_{\beta' \geq \beta_n}D(\beta')$.
\end{enumerate}
Now, we prove that $\{ f(y_\beta), \beta \in B \}$ converges in $\mathcal{T}_{\varepsilon,\lambda}$ to $\xi$ as follows: since $\sup_{k \geq n} |f(x_k) - \xi |$ converges almost surely to 0 as $n \to \infty$, there exists $N(\varepsilon,\lambda) \in N$ for any given positive numbers $\varepsilon$ and $\lambda$ with $0 < \lambda < 1$ such that $\sup_{k \geq N(\varepsilon,\lambda)} |f(x_k) - \xi | \in \mathcal{U}_{\theta}(\varepsilon,\lambda)$, where $\mathcal{U}_{\theta}(\varepsilon,\lambda)$ denotes the $(\varepsilon,\lambda)$--neighborhood of $\theta$ in $L^0(\mathcal{F},K)$, then by (2) one has $y_\beta \in conv_{L^0} \{ x_n, n \geq N(\varepsilon,\lambda)\}$ for each $\beta \geq \beta_{N(\varepsilon,\lambda)}$, and thus $|f(y_\beta) - \xi | \leq \sup_{k \geq N(\varepsilon,\lambda)}|f(x_k)-\xi|$, which implies that $|f(y_\beta) - \xi| \in \mathcal{U}_{\theta}(\varepsilon,\lambda)$ for each $\beta \geq \beta_{N(\varepsilon,\lambda)}$, namely $f(y_\beta)$ converges in $\mathcal{T}_{\varepsilon,\lambda}$ to $\xi$, so that $\xi = f(y) \in f(G)$.

Then, we prove that $f(G)$ is $L^0$--convexly compact: in fact, let $\{ x_\alpha, \alpha \in \Gamma \}$ be any net in $G$, then by the $L^0$--convex compactness of $G$ there exists a subnet $\{ y_\beta, \beta \in B\}$ of $L^0$--convex combinations of $\{ x_\alpha, \alpha \in \Gamma \}$ such that $\{ y_\beta, \beta \in B\}$ converges in $\mathcal{T}_{\varepsilon,\lambda}$ to some $y \in G$, and hence $\{ f(y_\beta), \beta \in B \}$ converges in $\mathcal{T}_{\varepsilon,\lambda}$ to $f(y)$ for given $f$ in $E^{\ast}_{\varepsilon,\lambda}$. Clearly, $\{ f(y_\beta), \beta \in B \}$ is a subnet of $L^0$--convex combinations of $\{ f(x_\alpha), \alpha \in \Gamma \}$, so $f(G)$ is $L^0$--convexly compact by Proposition \ref{Proposition 2.12}.

Finally, by Theorem \ref{Theorem 2.17} $f(G)$ is almost surely bounded for any given $f \in E^{\ast}_{\varepsilon,\lambda}$, and further by Proposition \ref{Proposition 2.18} $G$ is almost surely bounded.
\end{proof}
\par
To give a most powerful characterization for a $\mathcal{T}_{\varepsilon,\lambda}$--closed $L^0$--convex subset to be $L^0$--convexly compact, namely Theorem \ref{Theorem 2.21} below, whose proof needs Lemma \ref{Lemma 2.20} below as well as a special case of Theorem \ref{Theorem 3.6} below whose proof is postponed to Section \ref{Section 3} of this paper.

Let $(E,\|\cdot\|)$ be a $\mathcal{T}_{\varepsilon,\lambda}$--complete $RN$ module over $K$ with base $(\Omega,\mathcal{F},P)$ and $1\leq p \leq +\infty$. Further, let $L^p(E) = \{ x \in E : \|x\|_p < +\infty \}$, where $\|x\|_p$ denotes the ordinary $L^p$--norm of $\|x\|$, namely $\|x\|_p = (\int_{\Omega} \|x\|^p dP)^{1/p}$ for $1\leq p < +\infty$ and $\|x\|_{\infty} = \inf\{ M \in [0,+\infty) ~|~ \|x\| \leq M \}$, then $(L^p(E),\|\cdot\|_p)$ is a Banach space over $K$. Since $(E^*, \|\cdot\|)$ is also a $\mathcal{T}_{\varepsilon,\lambda}$--complete $RN$ module, let $q$ be the H\"{o}lder conjugate number of $p$, then $(L^q(E^*),\|\cdot\|_q)$ is still a Banach space. As usual, $L^p(E)'$ denotes the classical conjugate space of $L^p(E)$.
\begin{lemma}\label{Lemma 2.20}\cite{Guo97}
Let $(E,\|\cdot\|)$ be a $\mathcal{T}_{\varepsilon,\lambda}$--complete $RN$ module over $K$ with base $(\Omega,\mathcal{F},P)$ and $p$ a positive number such that $1\leq p <+\infty$. Define the canonical mapping $T : L^q(E^*) \to L^p(E)'$ $($$T_f$ denotes $T(f)$ for each $f \in L^q(E^*)$$)$ by $T_f(x) = \int_{\Omega} f(x)dP$ for any $x \in L^p(E)$, then $T$ is an isometric isomorphism from $L^q(E^*)$ onto $L^p(E)'$.
\end{lemma}
\par
Let $(E,\|\cdot\|)$ be a $\mathcal{T}_{\varepsilon,\lambda}$--complete $RN$ module and $E^{**}$ its second random conjugate space, the canonical mapping $J: E \to E^{**}$ is defined by $J(x)(f) = f(x)$ for any $x \in E$ and $f \in E^*$, then $J$ is $L^0$--norm--preserving by the Hahn--Banach theorem for random linear functionals, if, in addition, $J$ is also surjective, then $(E,\|\cdot\|)$ is said to be random reflexive. In 1997, Guo proved in \cite{Guo97} that $(E,\|\cdot\|)$ is random reflexive if and only if $L^p(E)$ is reflexive for any given $p$ such that $1<p<+\infty$, which was further used by Guo and Li in 2005 in \cite{GL05} to prove that $(E,\|\cdot\|)$ is random reflexive if and only if each $f \in E^*$ can attain its $L^0$--norm on the random closed unit ball of $E$.
\par
Let $(B,\|\cdot\|)$ be a Banach space, the famous James' weak compactness determination theorem \cite{J64} says that a nonempty weakly closed subset $G$ of $B$ is weakly compact if and only if for each $f \in B^{\prime}$ there exists $g_0 \in G$  such that $Re(f(g_0)) = \sup \{ Re(f(g)) : g \in G \}$. Since when $G$ is convex, $G$ is weakly closed if and only if $G$ is closed, in which case the James theorem becomes : a closed convex subset $G$ of $B$ is weakly compact if and only if for each $f \in B'$ there exists $g_0 \in G$  such that $Re(f(g_0)) = \sup \{ Re(f(g)) : g \in G \}$, the special case can be generalized to a $\mathcal{T}_{\varepsilon,\lambda}$--complete $RN$ module as follows:
\begin{theorem}\label{Theorem 2.21}
Let $(E,\|\cdot\|)$ be a $\mathcal{T}_{\varepsilon,\lambda}$--complete $RN$ module over $K$ with base $(\Omega,\mathcal{F},P)$ and $G$ a $\mathcal{T}_{\varepsilon,\lambda}$--closed $L^0$--convex subset of $E$. Then $G$ is $L^0$--convexly compact if and only if for each $f \in E^*$ there exists $g_0 \in G$  such that $Re(f(g_0)) = \bigvee\{ Re(f(g)) : g \in G \}$.
\end{theorem}
\begin{proof}
Necessity. Define $\tilde{f} : G \to L^0(\mathcal{F})$ by $\tilde{f}(g) = -Re(f(g))$ for any $g \in G$, it is obvious that $\tilde{f}$ is stable, $L^0$--convex and $\mathcal{T}_{c}$--semicontinuous, so $\tilde{f}$ satisfies the condition of Theorem \ref{Theorem 3.6} below of this paper.
\par
Sufficiency. First, we assert that $G$ is a.s. bounded, it only needs to verify that $\{ |f(g)| : g \in G \}$ is a.s. bounded by the resonance theorem \cite{Guo96b,Guo13} or Proposition \ref{Proposition 2.18}. In fact, since for each $f \in E^*$ there exists $g_0 \in G$ such that $Re(f(g_0)) = \bigvee\{Re(f(g)) : g \in G \}$, $\{ Re(f(g)) : g \in G \}$ is bounded above by $Ref(g_0)$ in  $(L^0(\mathcal{F}),\leq)$, similar to proof of Lemma \ref{Lemma 2.19}, one can see that $G$ is a.s. bounded.
\par
Now, we can, without loss of generality, assume that there exists $\xi \in L^0_{++}(\mathcal{F})$ such that $\|g\| \leq \xi$ for any $g \in G$, we can further assume $\xi=1$ (since otherwise, we may first consider $\widetilde{G} := G/\xi := \{ g/\xi : g \in G \}$ by noticing that $G$ and $\widetilde{G}$ have the same $L^0$--convex compactness). Then it is very easy to verify that $G$ is a bounded closed convex subset of the Banach space $(L^2(E), \|\cdot\|_2)$, in fact, the $(\varepsilon,\lambda)$--topology and the $\|\cdot\|_2$--topology coincide on $G$ by the Lebesgue dominance convergence theorem. Next, we will prove $G$ is a weakly compact subset of $L^2(E)$.
\par
Let $F$ be any given continuous linear functional on $L^2(E)$, then by Lemma \ref{Lemma 2.20} there exists a unique $f \in L^2(E^*)$ such that $F(x) = \int_{\Omega} f(x) dP$ for any $x \in L^2(E)$ (and hence $Re(F(x)) = \int_{\Omega} Re(f(x))dP$). Since there exists $g_0 \in G$ such that $Re(f(g_0)) = \bigvee \{Re(f(g)) : g \in G \}$, then it is clear that $Re(F(g_0)) = \int_{\Omega} Re(f(g_0)) dP \geq \sup \{ \int_{\Omega} Re(f(g)) dP : g \in G \} = \sup \{ Re(F(g)) : g \in G \}$. On the other hand, $\{ Re(f(g)) : g \in G \}$ is directed upwards: for any $g_1$ and $g_2$ in $G$, let $A = (Re(f(g_1)) \leq Re(f(g_2)))$ and $g_3 = \tilde{I}_{A^c} g_1 + \tilde{I}_A g_2$, then $g_3 \in G$ and $Re(f(g_3)) = Re(f(g_1)) \bigvee Re(f(g_2))$. So, by Proposition \ref{Proposition 1.3} there exists a sequence $\{ g_n, n \in N \}$ in $G$ such that $\{ Re(f(g_n)), n \in N \}$ converges a.s. to $\bigvee \{ Re(f(g)) : g \in G \} = Re(f(g_0))$ in a nondecreasing fashion, further by noticing that $|Re(f(g_n))| \leq |f(g_n)| \leq \|f\|$ for each $n \in N$ and $\int_{\Omega} \|f\| dP \leq \|f\|_2 < +\infty$, we can have that $Re(F(g_0)) = \int_{\Omega} Re(f(g_0)) dP = \lim_{n\to\infty} \int_{\Omega} Re(F(g_n)) dP \leq \sup\{ \int_{\Omega} Re(f(g)) dP : g \in G \}= \sup\{ Re(F(g)) : g \in G \}$. To sum up, $G$ is weakly compact by the classical James theorem, which is also equivalent to saying that $G$ is convexly compact, $G$ is ,of course, $L^0$--convexly compact.
\end{proof}
\begin{remark}\label{Remark 2.22}
By definition, for a $\mathcal{T}_{\varepsilon,\lambda}$--closed $L^0$--convex subset $G$ of a $\mathcal{T}_{\varepsilon,\lambda}$--complete $RN$ module $(E,\|\cdot\|)$, its convex compactness obviously implies its $L^0$--convex compactness, but the process of proof of Theorem \ref{Theorem 2.21} shows that the converse is also true by proving that
$G$ is linearly homeomorphic to a convexly compact subset $\widetilde{G}$ of the Banach space $L^2(E)$. Recently, Wu and Zhao \cite{WZ19} have extended the equivalence to the context of a complete random locally convex module.
\end{remark}

For a complex number $z\neq 0$, $arg(z)$ denotes the principal argument of $z$, we specify $arg(z) \in [0, 2\pi)$, whereas we make the convention $arg(z) = 2\pi$ when $z = 0$. Let $\xi \in L^0(\mathcal{F},K)$ with a representation $\xi^0$, then $arg(\xi^0(\cdot))$ is a real--valued random variable, if we use $arg(\xi)$ for the equivalence class of $arg(\xi^0(\cdot))$, then $\xi = |\xi| e^{iarg(\xi)}$.

\par
\begin{corollary}\label{Corollary 2.23}
Let $(E,\|\cdot\|)$ be a $\mathcal{T}_{\varepsilon,\lambda}$--complete random normed module over $K$ with base $(\Omega,\mathcal{F},P)$. Then $E$ is random reflexive iff every $\mathcal{T}_{\varepsilon,\lambda}$--closed, $L^0$--convex and almost surely bounded subsets of $E$ is $L^0$--convexly compact.
\end{corollary}
\begin{proof}
(1). Necessity. Since $E$ is random reflexive, it follows from \cite{Guo97,GL05} that $(L^2(E),\|\cdot\|_2)$ is a reflexive Banach space, where $L^2(E) = \{ x \in E ~|~ \int_{\Omega} \|x\|^{2} dP < +\infty \}$ and $\|x\|_{2} = (\int_{\Omega} \|x\|^{2}dP)^{\frac{1}{2}}$ for all $x \in L^2(E)$. Now, let $G$ be a $\mathcal{T}_{\varepsilon,\lambda}$--closed, $L^0$--convex and almost surely bounded subset of $E$ and further let $\xi \in L^0_{++}(\mathcal{F})$ such that $\|g\| \leq \xi$ for all $g \in G$. We can, without loss of generality, suppose that $\xi = 1$ (otherwise, we can consider $\frac{1}{\xi}G$ in the place of $G$). Then $G$ is a closed convex subset of the closed unit ball $\{ x \in L^2(E) ~|~ \|x\|_2 \leq 1 \}$, and hence a weakly compact set of $L^2(E)$, which, of course, implies that $G$ is $L^0$--convexly compact.

(2). Sufficiency. Let $U(1) = \{ x \in E : \|x\| \leq 1 \}$, then for each $f \in E^*$ there exists some $g_0 \in U(1)$ such that $Re(f(g_0)) = \bigvee \{ Re(f(g)) : g \in U(1)\}$. Since for each $x \in E$, $f(x) = |f(x)| e^{iarg(f(x))}$, then $|f(x)| =f(x) \cdot e^{-iarg(f(x))} = f(e^{-iarg(f(x))} \cdot x) = Re(f(e^{-iarg(f(x))} \cdot x))$, from which one can easily see that $\bigvee \{ Re(f(g)) : g \in U(1) \} = \bigvee \{ |f(g)| : g \in U(1) \}$, it is , of course, that $Re(f(g_0)) = |f(g_0)|$, namely $f(g_0) = Re(f(g_0))$. To sum up, we have that $f(g_0)= \bigvee \{ |f(g)| : g \in U(1) \} = \|f\|$.¡£ It follows from \cite[Theorem 3.1]{GL05} that $E$ is random reflexive.
\end{proof}

For any positive integer $d$, let $L^0(\mathcal{F},K^d)$ be the $L^0(\mathcal{F},K)$--module of equivalence classes of $K^d$--valued random vectors on $(\Omega,\mathcal{F},P)$, where $K^d$ is the Cartesian product of $K$ by $d$ times, it is a free $L^0(\mathcal{F},K)$--module of rank $d$. As is well known, any finitely generated subspace of a linear space over $K$ must be of finite dimension, for example, simple like some $K^d$. However, a finitely generated submodule of an $L^0(\mathcal{F},K)$--module is rather different from a free $L^0(\mathcal{F},K)$--module of finite rank, whose structure was already characterized in \cite[Theorem 1.1]{GS11}. Since we are often forced to work with finitely generated $L^0(\mathcal{F},K)$--modules rather than free $L^0(\mathcal{F},K)$--modules of finite rank, (see Section \ref{Section 4} of this paper), we restate it for the sake of convenience.
\begin{proposition}\label{Proposition 2.24}  \cite[Theorem 1.1]{GS11}
Let $E$ be a finitely generated $L^0(\mathcal{F},K)$--module, e.g., let $E = span_{L^0} \{ p_1, p_2,...,p_n \} := \{ \sum^{n}_{i = 1} \xi_{i}p_{i} ~|~ \xi_1, \xi_2,...,\xi_n \in L^0(\mathcal{F},K) \}$ for $n$--fixed elements $p_1,p_2,...,p_n \in E$. Then there exists a finite partition $\{ A_0,A_1,...,A_n \}$ of $\Omega$ to $\mathcal{F}$ such that $\tilde{I}_{A_i} E$ is a free module of rank $i$ over the algebra $\tilde{I}_{A_i} L^0(\mathcal{F},K)$ for any $i \in \{ 0,1,2,...,n \}$ such that $P(A_i) > 0$, in which case $E$ has the direct sum decomposition as $\bigoplus^{n}_{i=0} \tilde{I}_{A_i} E$ and each such $A_i$ is unique up to the almost sure equality.
\end{proposition}

\begin{corollary}\label{Corollary 2.25}
Let $(E,\mathcal{P})$ be a finitely generated random locally convex module over $K$ with base $(\Omega, \mathcal{F},P)$, for example, let $E = \bigoplus^{n}_{i=0} \tilde{I}_{A_i}E$ be the same as in Proposition \ref{Proposition 2.24} $($we can, without loss of generality, assume $P(A_i) > 0$ for each $i$ with $0 \leq i \leq n)$. Then $(E,\mathcal{T}_{\varepsilon,\lambda})$ is isomorphic onto a closed submodule of $(L^0(\mathcal{F},K^n),\mathcal{T}_{\varepsilon,\lambda})$ in the sense of topological modules. In particular, any $\mathcal{T}_{\varepsilon,\lambda}$--closed, almost surely bounded and $L^0$--convex subset $G$ of $E$ must be $L^0$--convexly compact. Here, $L^0(\mathcal{F},K^n)$ is endowed with the $L^0$--inner product $\langle \xi, \eta \rangle = \sum^{n}_{i=1} \xi_{i} \overline{\eta}_i$ for any $\xi = (\xi_1,\xi_2,...,\xi_n)^{T}$ and $\eta = (\eta_1,\eta_2,...,\eta_n)^T \in L^0(\mathcal{F},K^n)$, where the symbol $T$ stands for the transpose operation of a vector.
\end{corollary}
\begin{proof}
By Lemma 3.4 of \cite{GP01}, each $\tilde{I}_{A_i} \cdot E$ is isomorphic onto $\tilde{I}_{A_i} L^0(\mathcal{F},K^i)$ in the sense of a topological module for each $i$ with $1 \leq i \leq n$ (we can omit $\tilde{I}_{A_0}E$ since it is $\{ \theta \}$). Since $L^0(\mathcal{F},K^i)$ can be identified with $\{ \xi \in L^0(\mathcal{F},K^n) ~|~ \xi_k = 0$ when $k \geq i+1 \}$ for each $i$ such that $1 \leq i \leq n-1$, $E$ is isomorphic onto $\sum^{n}_{i=1} \tilde{I}_{A_i} L^0(\mathcal{F},K^i)$, a closed submodule of $(L^0(\mathcal{F},K^n),\mathcal{T}_{\varepsilon,\lambda})$.

Denote by $J$ the above isomorphism from $E$ onto $\sum^{n}_{i=1} \tilde{I}_{A_i} L^0(\mathcal{F},K^i)$, then $J(E)$ is random reflexive since $L^0(\mathcal{F},K^n)$ is random reflexive. Further, since $G$ is almost surely bounded, it is, of course, $\mathcal{T}_{\varepsilon,\lambda}$--bounded, then $J(G)$ is also $\mathcal{T}_{\varepsilon,\lambda}$--bounded, which further implies that $J(G)$ is almost surely bounded since $J(G)$ is $L^0$--convex. To sum up, $J(G)$ is $L^0$--convexly compact by Corollary \ref{Corollary 2.23} since it is $\mathcal{T}_{\varepsilon,\lambda}$--closed, almost surely bounded and $L^0$--convex, this means $G$ is $L^0$--convexly compact, too.
\end{proof}

The randomized Bolzano--Weierstrass theorem (see, e.g.\cite{Yan04}) can be stated as follows: for each a.s. bounded sequence $\{x_n : n \in N \}$ in $L^0(\mathcal{F},R^n)$ there exists a sequence $\{ n_k : k \in N \}$ of positive integer--valued random variables such that $n_k(\omega) < n_{k+1}(\omega)$ for each $\omega \in \Omega$ and each $k \in N$ such that $\lim_{k \to \infty} n_{k}(\omega) = +\infty$  for each $\omega \in \Omega$ and $\{ x_{n_k} : k \in N \}$ converges a.s.to some element $y$ in $L^0(\mathcal{F},R^n)$, where  $x_{n_k} = \sum^{\infty}_{l=1} \tilde{I}_{A_{k,l}} \cdot x_l$ for each $k \in N$ with $A_{k,l} = \{ \omega \in \Omega : n_k(\omega) = l \}$ for each $l \in N$. Corollary \ref{Corollary 2.25} tells us that each a.s. bounded $\mathcal{T}_{\varepsilon,\lambda}$--closed $L^0$--convex subset of $L^0(\mathcal{F},R^n)$ is $L^0$--convexly compact. Proposition \ref{Proposition 2.26} below shows that the randomized Bolzano--Weierstrass theorem implies the result of $L^0$--convex compactness.

\begin{proposition}\label{Proposition 2.26}
 The randomized Bolzano--Weierstrass theorem implies that each a.s. bounded $\mathcal{T}_{\varepsilon,\lambda}$--closed $L^0$--convex subset $G$ of $L^0(\mathcal{F},R^n)$ is $L^0$--convexly compact.
\end{proposition}
\begin{proof}
For any given sequence $\{x_n : n \in N \}$ in $G$, we can, without loss of generality, assume that $\theta \in G$. Let  $\{ n_k : k \in N \}$ and $\{ x_{n_k} : k \in N \}$ be as above, we can also, without loss of generality, suppose that $n_k(\omega) >k$ for each $k \in N$ and $\omega \in \Omega$. It is clear that we can select a sufficiently great $m_k \in N$ for each $k \in N$ such that $\sum_{l > m_k} P(n_k =l) < \frac{1}{k}$, where $(n_k =l) = \{ \omega \in \Omega : n_k(\omega) = l \}$ for each positive integer $l > m_k$. Further, let $y_k = \sum^{m_k}_{l=k+1} \tilde{I}_{A_{k,l}} \cdot x_l$ for each $k \in N$, where $A_{k,l} = \{ \omega \in \Omega : n_k(\omega) = l \}$ for each $l$ such that $k+1 \leq l \leq m_k$, then it is also obvious that $y_k \in conv_{L^0}\{ x_l : l > k \}$ and $y_k \in G$ for each $k \in N$ satisfying that $\{ y_k : k \in N \}$ converges in probability (namely, in $\mathcal{T}_{\varepsilon,\lambda}$) to $y$, which shows that $G$ is $L^0$--convexly compact by Proposition \ref{Proposition 2.12}.

\end{proof}
\section{Attainment of infima and Minty type variational inequalities for $L^0$--convex functions}\label{Section 3}
\par
In this section, $(E,\mathcal{P})$ always denotes a given random locally convex module over the real number field $R$ with base $(\Omega,\mathcal{F},P)$ and $G$ is a $\mathcal{T}_{\varepsilon,\lambda}$--closed $L^0$--convex subset of $E$.
\par
The main results of this section are Lemma \ref{Lemma 3.3}, Theorem \ref{Theorem 3.5}, Theorem \ref{Theorem 3.6}, Theorem \ref{Theorem 3.8}, Theorem \ref{Theorem 3.15}, Theorem \ref{Theorem 3.16} and Corollary \ref{Corollary 3.17}, let us first recapitulate some known terminology for the statement and proof of them.
\par
$L^0$--convex and $L^0$--quasiconvex functions defined on the whole space were already studied in \cite{FKV09,GZZ15a,GZZ15b,GZW17,FM14a,FM14b}, Definition \ref{Definition 3.1} below will be convenient for us in this paper.

\begin{definition}\label{Definition 3.1}
A mapping $f: G\to \bar{L}^0(\mathcal{F})$ is said to be
\begin{enumerate}[(1)]
\item an $L^0$--convex function if $f(\xi x+(1-\xi)y)\leq\xi f(x)+(1-\xi)f(y)$ for all $x$ and $y\in G$ and $\xi \in L^0_+(\mathcal{F})$ such that $0\leq\xi\leq 1$, where we adopt the convention that $0\cdot(\pm\infty)=0, +\infty\pm(-\infty)=+\infty$.
\item a local function if $\theta$ $($ the null in $E$ $)$ $\in G$ and $\tilde{I}_Af(x)=\tilde{I}_Af(\tilde{I}_Ax)$ for all $x\in G$ and all $A\in \mathcal{F}$.
\item stable or regular if $f(\tilde{I}_Ax+\tilde{I}_{A^c}y)=\tilde{I}_Af(x)+\tilde{I}_{A^c}f(y)$ for all $x$ and $y\in G$ and all $A\in \mathcal{F}$, where $A^c=\Omega\setminus A$, namely the complement of $A$.
\item $\sigma$--stable or countably regular if $G$ has the countable concatenation property and $f(\sum^{\infty}_{n=1}\tilde{I}_{A_n}x_n)=\sum^{\infty}_{n=1}\tilde{I}_{A_n}f(x_n)$ for all sequences $\{x_n\}^{\infty}_{n=1}$ in $G$ and all countable partitions $\{A_n\}^{\infty}_{n=1}$ of $\Omega$ to $\mathcal{F}$.
\item proper if $f(x)>-\infty$ on $\Omega$ for all $x\in G$ and there exists some $x\in G$ such that $f(x)\in L^0(\mathcal{F})$.
\item $\mathcal{T}_{\varepsilon,\lambda}$--lower semicontinuous if $f$ is proper and epi$(f)=\{(x,r)\in G\times L^0(\mathcal{F})~|~f(x)\leq r\}$ is closed in $(E,\mathcal{T}_{\varepsilon,\lambda})\times(L^0(\mathcal{F}),\mathcal{T}_{\varepsilon,\lambda})$.
\item $\mathcal{T}_c$--lower semicontinuous if $f$ is proper and $\{x\in G~|~f(x)\leq \eta\}$ is $\mathcal{T}_c$--closed for all $\eta\in L^0(\mathcal{F})$.
\item $L^0$--quasiconvex if $f$ is proper and $\{x\in G~|~f(x)\leq\eta\}$ is $L^0$--convex, $\forall \eta\in L^0(\mathcal{F})$.
\end{enumerate}
\end{definition}

\begin{remark}\label{Remark 3.2}
In Definition \ref{Definition 3.1}, if $G$ contains the null $\theta$ of $E$ and has the countable concatenation property, then the stability, $\sigma$--stability and the local property of $f$ coincide. Generally, if $G$  has the countable concatenation property, then it is easy to verify that the stability of $f$ can also imply the $\sigma$--stability of $f$.
\end{remark}

\par
Although an element $x\in E$ may not belong to $G$, it is possible that there exists some $A\in \mathcal{F}$ such that $\tilde{I}_Ax\in\tilde{I}_AG: =\{\tilde{I}_Ag~|~g\in G\}$, we are often interested in esssup$\{A\in \mathcal{F}~|~\tilde{I}_Ax\in \tilde{I}_AG\}$, as shown for the family $\mathcal{E}$ in the proof of \cite[Theorem 3.13]{Guo10}.

\begin{lemma}\label{Lemma 3.3}
Let $\mathcal{E}(x,G)=\{A\in \mathcal{F}~|~\tilde{I}_Ax\in\tilde{I}_AG\}$ and $S(x,G)=esssup(\mathcal{E}(x,G))$ for any $x\in E$. Then we have the following statements:
\begin{enumerate}[(1)]
\item $\mathcal{E}(x,G)$ is directed upwards $($in fact, is closed under the finite union operation$)$ for any fixed $x\in E$.
\item If $\theta\in G$, then $\mathcal{E}(x,G)=\{A\in\mathcal{F}~|~\tilde{I}_Ax\in G\}$ for any fixed $x\in E$.
\item If $\theta\in G$, then $S(x,G)\in\mathcal{E}(x,G)$ for any fixed $x\in E$.
\item If $\theta\in G$, then $S(\tilde{I}_Bx,G)=B^c\cup S(x,G)$ for any fixed $x\in E$ and $B\in\mathcal{F}$.
\end{enumerate}
\end{lemma}
\begin{proof}
$(1)$. Its proof is omitted since this proof is the same as that of $\mathcal{E}$ in \cite[Theorem 3.13]{Guo10}.
\par
$(2)$. The proof of $(2)$ is obvious.
\par
$(3)$. By $(1)$, there exists a nondecreasing sequence $\{A_n~|~n\in N\}$ in $\mathcal{E}(x,G)$ such that $\bigcup_{n\in N}A_n=S(x,G)$, then $\tilde{I}_{S(x,G)}x=\lim_{n\to \infty}\tilde{I}_{A_n}x\in G$ since each $\tilde{I}_{A_n}x\in G$ and $G$ is $\mathcal{T}_{\varepsilon,\lambda}$--closed, which in turn implies that $\tilde{I}_{S(x,G)}x=\tilde{I}_{S(x,G)} \cdot (\tilde{I}_{S(x,G)}x)\in\tilde{I}_{S(x,G)}G$, namely $S(x,G)\in \mathcal{E}(x,G)$.
\par
$(4)$. It is easy to see that both $B^c$ and $S(x,G)$ belong to $\mathcal{E}(\tilde{I}_Bx,G)$, so $B^c\cup S(x,G)\subset S(\tilde{I}_Bx,G)$. On the other hand, if $A\in \mathcal{E}(\tilde{I}_Bx,G)$, then $A\cap B\subset S(x,G)$, namely $A\cap B\cap S(x,G)^c=\emptyset $, which is equivalent to $A\subset B^c\cup S(x,G)$, and hence $S(\tilde{I}_Bx,G)\subset B^c\cup S(x,G)$. To sum up, we have that $S(\tilde{I}_Bx,G)=B^c\cup S(x,G)$.
\end{proof}

\par
In classical convex analysis, it is quite easy to extend a convex function defined on a closed convex subset to one defined on the whole space, see \cite[p.34]{ET99}, whereas it is completely another matter to extend an $L^0$--convex function defined on a $\mathcal{T}_{\varepsilon,\lambda}$--closed $L^0$--convex subset $G$ to one defined on the whole $L^0(\mathcal{F},K)$--module $E$. Now, on the basis of Lemma \ref{Lemma 3.3}, Lemma \ref{Lemma 3.4} below arrives at this aim, in particular when $f: G\to \bar{L}^0(\mathcal{F})$ is the constant function with value 0 on $G$. We can obtain a special function $\bar{f}: E\to \bar{L}^0(\mathcal{F})$ given by $\bar{f}(x)=\tilde{I}_{S(x,G)^c}(+\infty)$ for any $x\in E$, denoted by $\mathcal{X}_G$ and called the indicator function of $G$. By the way, Lemma \ref{Lemma 3.4} and $\mathcal{X}_G$ will play a crucial role in the section and in particular in the next section of this paper.

\begin{lemma}\label{Lemma 3.4}
Let $\theta\in G$ and $f: G\to \bar{L}^0(\mathcal{F})$ be a mapping. Define the mapping $\bar{f}: E\to \bar{L}^0(\mathcal{F})$ by $\bar{f}(x)=\tilde{I}_{S(x,G)}f(\tilde{I}_{S(x,G)}x)+\tilde{I}_{S(x,G)^c}(+\infty)$. Then the following statements hold:
\begin{enumerate}[(1)]
\item $\bar{f}$ is an extension of $f$.
\item $f$ is local iff $\bar{f}$ is local.
\item $f$ is proper iff $\bar{f}$ is proper.
\item If $f$ is proper, then $f$ is $L^0$--convex iff $f$ is local and epi$(f)$ is $L^0$--convex.
\item If $f$ is proper, then $f$ is $L^0$--convex iff $\bar{f}$ is $L^0$--convex.
\item $f$ is proper and $\mathcal{T}_{\varepsilon,\lambda}$--lower semicontinuous iff $\bar{f}$ is proper  and $\mathcal{T}_{\varepsilon,\lambda}$--lower semicontinuous.
\item $f$ is proper and $\mathcal{T}_c$--lower semicontinuous iff $\bar{f}$ is proper  and $\mathcal{T}_c$--lower semicontinuous.
\end{enumerate}
\end{lemma}
\begin{proof}
$(1)$. is obvious.
\par
$(2)$. The locality of $\bar{f}$ obviously implies the locality of $f$. Conversely, let $f$ be local, then, for each $B\in \mathcal{F}$ and $x\in E$, by definition: $\bar{f}(x)=\tilde{I}_{S(x,G)}f(\tilde{I}_{S(x,G)}x)+\tilde{I}_{S(x,G)^c}(+\infty)$ and $\bar{f}(\tilde{I}_Bx)=\tilde{I}_{S(\tilde{I}_Bx,G)}f(\tilde{I}_{S(\tilde{I}_Bx,G)}\cdot\tilde{I}_Bx)+\tilde{I}_{S(\tilde{I}_Bx,G)^c}(+\infty)$; according to $(4)$ of Lemma \ref{Lemma 3.3}, $S(\tilde{I}_Bx,G)=B^c\cup S(x,G)$, then $\bar{f}(\tilde{I}_Bx)=\tilde{I}_{B^c\cup S(x,G)}f(\tilde{I}_{B\cap S(x,G)}x)+\tilde{I}_{B\cap S(x,G)^c}(+\infty)$, which further implies that $\tilde{I}_B\bar{f}(\tilde{I}_Bx)=\tilde{I}_{B\cap S(x,G)}f(\tilde{I}_{S(x,G)}x)+\tilde{I}_{B\cap S(x,G)^c}(+\infty)=\tilde{I}_B\bar{f}(x)$.
\par
$(3)$. It is clear by observing $dom(\bar{f}): =\{x\in E~|~\bar{f}(x)<+\infty$ on $\Omega\}=\{x\in G~|~f(x)<+\infty$ on $\Omega\}: =dom(f)$
\par
$(4)$. Its proof is omitted since its proof is completely similar to that of Theorem 3.2 of \cite{FKV09}.
\par
$(5)$. The $L^0$--convexity of $\bar{f}$ obviously implies the $L^0$--convexity of $f$. Conversely, by $(4)$ the $L^0$--convexity of $f$ implies that both $f$ is local and $epi(f)$ is $L^0$--convex, which in turn means that $\bar{f}$ is local by $(2)$ and $epi(\bar{f})=epi(f)$ is also $L^0$--convex, again by Theorem 3.2 of \cite{FKV09} $\bar{f}$ is $L^0$--convex.
\par
$(6)$ and $(7)$ both are obvious.
\end{proof}

\begin{theorem}\label{Theorem 3.5}
Let $(E,\mathcal{P})$ be such that both $E$ and $\mathcal{P}$ have the countable concatenation property, $\theta\in G$ and $f: G\to \bar{L}^0(\mathcal{F})$ a proper and local function. Then the following are equivalent:
\begin{enumerate}[(1)]
\item $f$ is $\mathcal{T}_c$--lower semicontinuous.
\item $f$ is $\mathcal{T}_{\varepsilon,\lambda}$--lower semicontinuous.
\item $epi(f)$ is $\mathcal{T}_c$--closed in $(E,\mathcal{T}_c)\times (L^0(\mathcal{F}),\mathcal{T}_c)$.
\item $\underline{\lim}_{\alpha}f(x_{\alpha})\geq f(x)$ for any $x\in G$ and any net $\{x_{\alpha},\alpha\in \Gamma\}$ in $G$ such that $\{x_{\alpha},\alpha\in \Gamma\}$ converges in $\mathcal{T}_c$ to $x$, where $\underline{\lim}_{\alpha}f(x_{\alpha})=\bigvee_{\beta\in\Gamma}(\bigwedge_{\alpha\geq\beta}f(x_{\alpha}))$ .
\end{enumerate}
\end{theorem}
\begin{proof}
Let $\bar{f}$ be the extension of $f$ as in Lemma \ref{Lemma 3.4}, then Theorem 2.13 of \cite{GZW17} shows that $(1), (2), (3)$ and $(4)$ are equivalent to another for $\bar{f}$, so they are still equivalent for $f$ by Lemma \ref{Lemma 3.4}.
\end{proof}

\par
In the sequel of this paper, for the sake of convenience we adopt the following convention: let $\xi$ and $\eta$ be in $\bar{L}^0(\mathcal{F})$ and arbitrarily choose $\xi^0$ and $\eta^0$ as representatives of $\xi$ and $\eta$ respectively, since $A=\{\omega\in \Omega ~|~ \xi^0(\omega)<\eta^0(\omega)\}$ is unique up to a set of zero probability, we briefly write $(\xi<\eta)$ for $A$, similarly one can understand such symbols as $(\xi\leq\eta), (\xi\neq\eta),(\xi=\eta)$ and so on.

Regarding Theorem \ref{Theorem 3.6} below, we would like to remind the reader of the fact that $G$ must have the countable concatenation property if $G$ is a $\mathcal{T}_{\varepsilon,\lambda}$--closed $L^0$--convex subset of a random locally convex module $(E,\mathcal{P})$ such that $E$ has the countable concatenation property.

\begin{theorem}\label{Theorem 3.6}
Let $(E,\mathcal{P})$ be a random locally convex module over $R$ with base $(\Omega,\mathcal{F},P)$ such that $E$ has the countable concatenation property, $G$ a nonempty $\mathcal{T}_{\varepsilon,\lambda}$--closed $L^0$--convex subset of $E$, $f : G\to \bar{L}^0(\mathcal{F})$ a proper, stable, $\mathcal{T}_c$--lower semicontinuous and $L^0$--quasiconvex function. If $G$ is $L^0$--convexly compact, then there exists $y_0\in G$ such that $f(y_0)=\bigwedge\{f(x)~|~x\in G\}$.
\end{theorem}

\begin{proof}
We can, without loss of generality, assume that $\mathcal{P}$ has the countable concatenation property, otherwise we consider the random locally convex module $(E, \mathcal{P}_{cc})$, where $\mathcal{P}_{cc}$ denotes the countable concatenation hull of $\mathcal{P}$, one only needs to bear in mind that $\mathcal{P}_{cc}$ and $\mathcal{P}$ induce the same $(\varepsilon,\lambda)$--topology and the locally $L^0$--convex topology induced by $\mathcal{P}_{cc}$ is stronger than that induced by $\mathcal{P}$. We can also, without loss of generality, assume $\theta\in G$ $($otherwise, we make a translation$)$, and let $\eta=\bigwedge\{f(x)~|~x\in G\}$. First, $\{f(x)~|~x\in G\}$ is directed downwards: for any $x_1$ and $x_2\in G$, let $A=(f(x_1)\leq f(x_2))$ and $x_3=\tilde{I}_Ax_1+ \tilde{I}_{A^c}x_2$, then $x_3\in G$ and $f(x_3)=\tilde{I}_Af(x_1)+\tilde{I}_{A^c}f(x_2)=f(x_1)\bigwedge f(x_2)$ by the stability of $f$. Thus, by Proposition \ref{Proposition 1.3} there exists a sequence $\{x_n,n\in N\}$ in $G$ such that $\{f(x_n),n\in N\}$ converges to $\eta$ in a nonincreasing way.
\par
By the $L^0$--convex compactness of $G$ and Theorem \ref{Theorem 2.16} there exists a net $\{y_{(n,U)}, (n,U)\in N\times \mathcal{U}\}$ convergent in $\mathcal{T}_c$ to some $y_0\in E$ such that the following two conditions are satisfied:
\begin{enumerate}[(i)]

\item $\mathcal{U}$ is a local base at $\theta$ of $\mathcal{T}_c$;
\item $y_{(n,U)}\in (y_0+ U)\cap H_{cc}(conv_{L^0}\{x_k, k\geq n\})$ for each $(n,U)\in N\times \mathcal{U}$.
\end{enumerate}
\par
Since $E$ has the countable concatenation property, it is easy to check that  the $\mathcal{T}_{\varepsilon,\lambda}$--closedness and $L^0$--convexity of $G$ imply that $G$ has the countable concatenation property so that each $y_{(n,U)}\in G$. Further, $y_0\in G$ since $G$ is also $\mathcal{T}_c$--closed by \cite[Theorem 3.12]{Guo10}.
\par
By the $L^0$--quasiconvexity of $f$, $f(y)\leq \bigvee_{k\geq n}f(x_k)$ for each $y\in conv_{L^0}\{ x_k,
 k \geq n \}$. Further, since $f$ is $\sigma$--stable by the fact that $\theta\in G$ and $f$ is local at this time, then $f(y)\leq \bigvee_{k\geq n}f(x_k)$ for each $y\in H_{cc}(conv_{L^0}\{x_k, k\geq n\})$. Thus, $f(y_{(n,U)})\leq \bigvee_{k\geq n}f(x_k)=f(x_n)$. Now, by the $\mathcal{T}_c$--lower semicontinuity of $f$ and Theorem \ref{Theorem 3.5}, $f(y_0)\leq \underline{\lim}f(y_{(n,U)})\leq \lim_n f(x_n)=\eta$, which means that $f(y_0)=\bigwedge\{f(x)~|~x\in G\}$.
\end{proof}

\begin{definition}\label{Definition 3.7}
Let $(E,\|\cdot\|)$ be a random normed module over $R$ with base $(\Omega,\mathcal{F},P)$ and $G\subset E$. A mapping $f : G\to \bar{L}^0(\mathcal{F})$ is coercive if $\{f(u_n),n\in N\}$ converges almost surely to $+\infty$ on $A$ for any sequence $\{u_n,n\in N\}$ in $G$ and any $A\in \mathcal{F}$ of positive probability such that $\{\|u_n\|,n\in N\}$ converges almost surely to $+\infty$ on $A$.
\end{definition}

\begin{theorem}\label{Theorem 3.8}
Let $(E,\|\cdot\|)$ be a random reflexive random normed module over $R$ with base $(\Omega,\mathcal{F},P)$, $G\subset E$ a nonempty $\mathcal{T}_{\varepsilon,\lambda}$--closed $L^0$--convex subset of $E$ and $f : G\to \bar{L}^0(\mathcal{F})$ a proper, stable, coercive, $\mathcal{T}_c$--lower semicontinuous and $L^0$--quasiconvex function. Then there exists $y_0\in G$ such that $f(y_0)=\bigwedge\{f(x)~|~x\in G\}$.
\end{theorem}

\begin{proof}
Denote $\bigwedge\{f(x)~|~x\in G\}$ by $\eta$, then as in the proof of Theorem \ref{Theorem 3.6} there exists a sequence $\{x_n,n\in N\}$ in $G$ such that $\{f(x_n),n\in N\}$ converges to $\eta$ in a nonincreasing way. Since $f$ is proper, we can, without loss of generality, assume that $f(x_1)\in L^0(\mathcal{F})$.
\par
First, we can assert that $\{x_n,n\in N\}$ is almost surely bounded. Otherwise, there exists some $A\in\mathcal{F}$ with positive probability such that $\bigvee_{n\in N}\|x_n\|=+\infty$ on $A$. To produce a contradiction, we prove that there exists a sequence $\{x^*_n,n\in N\}$ in $G$ with the following properties:
\begin{enumerate}[(1)]
\item for each $n\in N$ there exists a finite partition $\{B_k,1\leq k\leq n\}$ of $\Omega$ to $\mathcal{F}$ such that $x^*_n=\sum^n_{k=1}\tilde{I}_{B_k}x_k$;
\item $\|x^*_n\|=\bigvee^n_{k=1}\|x_k\|$ for each $n\in N$.
\end{enumerate}
\par
In fact, let $A_1=(\|x_1\|\leq\|x_2\|), A_2=A^c_1$ and $x^*_2=\tilde{I}_{A_1}x_2+\tilde{I}_{A_2}x_1$, then it is easy to check that $\|x^*_2\|=\|x_1\|\vee\|x_2\|$. By noting that $(1)$ and $(2)$ automatically hold when $n=1($by taking $x^*_1=x_1)$, that is to say, we have proved the above assertion for $n\leq 2$. Let the assertion hold for $n=k(k>2)$, then there exists a finite partition $\{B'_j,1\leq j\leq k\}$ of $\Omega$ to $\mathcal{F}$ and $x^*_k\in G$ such that $x^*_k=\sum^k_{j=1}\tilde{I}_{B^{\prime}_j}x_j$ and $\|x^*_k\|=\bigvee^k_{j=1}\|x_j\|$. Now, let $A=(\|x^*_k\|\leq\|x_{k+1}\|)$ and $x^*_{k+1} =\tilde{I}_Ax_{k+1}+\tilde{I}_{A^c}x^*_k$, then it is easy to see that $\|x^*_{k+1}\|=\|x_{k+1}\|\vee\|x^*_k\|=\bigvee^{k+1}_{j=1}\|x_j\|$ and $x^*_{k+1}=\sum^{k+1}_{j=1}\tilde{I}_{B_j}x_j$ with $B_{k+1}=A$ and $B_j=B^{\prime}_j\cap A^c$ for all $j$ such that $1\leq j \leq k$. So the induction method can be used to end the proof of this assertion.
\par
By the coercivity of $f$, $\{f(x^*_n), n\in N\}$ converges almost surely to $+\infty$ on $A$, but $f(x^*_n)=\sum^n_{k=1}\tilde{I}_{A_k}f(x_k)\leq f(x_1)$ by the stability of $f$, which contradicts to the assumption on $f(x_1)$.
\par
Setting $\xi=\bigvee_{k\geq 1}\|x_k\|$ and $G_1=G\cap \{x\in G~|~\|x\|\leq\xi$ and $f(x)\leq f(x_1)\}$, then $G_1$ is both $\mathcal{T}_{\varepsilon,\lambda}$--closed and $\mathcal{T}_c$--closed since $G_1$ has the countable concatenation property $($since $E$ has the property by the $\mathcal{T}_{\varepsilon,\lambda}$--completeness of $E)$. Further, by Corollary \ref{Corollary 2.23} $G_1$ is $L^0$--convexly compact. Since $\mathcal{P}=\{\|\cdot\|\}$, of course, has the countable concatenation property, by Theorem \ref{Theorem 3.6} there exists $y_0\in G_1$ such that $f(y_0)=\bigwedge\{f(x)~|~x\in G_1\}$. Finally, one can obviously observe that each $x_n$ is in $G_1$, then $f(y_0)\leq\bigwedge\{f(x_n),n\in N\}=\eta=\bigwedge\{f(x)~|~x\in G\}$, namely $f(y_0)=\bigwedge\{f(x)~|~x\in G\}$.
\end{proof}

\begin{remark}\label{Remark 3.9}
Let $(E,\mathcal{P})$ be a random locally convex module over $R$ with base $(\Omega,\mathcal{F},P)$ and $G\subset E$ an $L^0$--convex subset of $E$. An $L^0$--convex function $f: G\to \bar{L}^0(\mathcal{F})$ is strictly $L^0$--convex if $f(\lambda x+(1-\lambda)y)<\lambda f(x)+(1-\lambda)f(y)$ on $(0<\lambda<1)\cap(x\neq y)$ for all $x,y\in G$ and $\lambda\in L^0_+(\mathcal{F})$ with $0\leq\lambda\leq 1$, where $(x\neq y)=(\bigvee_{\|\cdot\|\in\mathcal{P}}\|x-y\|>0)$. It is easy to see that if $f$ in Theorem \ref{Theorem 3.6} and Theorem \ref{Theorem 3.8} is strictly $L^0$--convex, then $y_0$ must be unique. Let $M$ be a $\mathcal{T}_{\varepsilon,\lambda}$--closed submodule of $L^2_{\mathcal{F}}(\mathcal{E})$, $\pi : M \to L^0(\mathcal{F})$ a continuous module homomorphism from $(M,\mathcal{T}_{\varepsilon,\lambda})$ to $(L^0(\mathcal{F}),\mathcal{T}_{\varepsilon,\lambda})$ such that there exists some $z_0 \in M$ such that $P\{ \omega \in \Omega : \pi(z_0)(\omega) \neq 0 \} =1, \mathrm{w} \in L^0(\mathcal{F})$,$G = \{ x \in M : \pi(x) =1 $ and $E[x|\mathcal{F}] = \mathrm{w} \}$ and $f : M \to L^0(\mathcal{F})$ be defined by $f(x) = D(x|\mathcal{F}) := E[|x|^2|\mathcal{F}] - (E[x|\mathcal{F}])^2$ for any $x \in M$. Hansen and Richard \cite{HR87} proved that there exists unique one $x \in G$ such that $f(x) = \bigwedge\{ f(y) : y \in G \}$. Clearly, the result is a spacial case of Theorem \ref{Theorem 3.8} since $f$ is a proper, stable, coercive, $\mathcal{T}_{c}$--lower semicontinuous and strictly $L^0$--convex function on $G$.
\end{remark}

\begin{corollary}\label{Corollary 3.10}
Let $(E,\|\cdot\|)$ and $G$ be the same as in Theorem \ref{Theorem 3.8}, $a(\cdot,\cdot) : E\times E\to L^0(\mathcal{F})$ a $\mathcal{T}_{\varepsilon,\lambda}$--continuous $L^0$--bilinear form such that $a(x,x)\geq\alpha\|x\|^2$ for all $x\in E($where $\alpha$ is some fixed element of $L^0_{++}(\mathcal{F}))$, then for any given $l\in E^*$ there exists a unique $u\in G$ which achieves minimum over $G$ of the $L^0$--convex function $F$ defined by $F(x)=a(x,x)-2l(x)$ for all $x\in E$.
\end{corollary}

\begin{proof}
It is omitted since it is similar to that of Remark 1.1 of \cite[Chapter II]{ET99}.
\end{proof}

\par
As is shown in \cite[Chapter II, Section 2]{ET99}, Minty type variational inequalities can characterize solutions of minimazation problems, we introduce the notion of G\^{a}teaux derivatives $($slightly more general than that in \cite{GZW17}$)$ to obtain an $L^0$--module version of Minty type variational inequalities.

\begin{definition}\label{Definition 3.11}
Let $(E,\mathcal{P})$ be a random locally convex module over $R$ with base $(\Omega,\mathcal{F},P)$ and $f : E\to \bar{L}^0(\mathcal{F})$ is a proper function, $f$ is G\^{a}teaux--differentiable at $x\in dom(f)$ if there exists $g\in E^*_{\varepsilon,\lambda}$ such that the almost sure limit of $\{\frac{f(x+\lambda_ny)-f(x)}{\lambda_n}, n\in N\}$ exists whenever $\{\lambda_n, n\in N\}$ is a sequence in $L^0_{++}(\mathcal{F})$ such that $\lambda_n\downarrow 0$ and $g(y)=$ the a.s--$\lim_{\lambda_n\downarrow 0}\frac{f(x+\lambda_ny)-f(x)}{\lambda_n}$ for all $y\in E$, at which time $g$ is also called the G\^{a}teaux--derivative of $f$ at $x$, denoted by $f'(x)$, $f$ is G\^{a}teaux--differentiable on a subset $G$ of $E$ if $f$ is G\^{a}teaux--differentiable at each point of $G$, in which case $f'$ is said to $H$--continuous on $G$ if for each $y\in E, f'(\cdot)(y):$ is sequentially continuous from $(G,\mathcal{T}_{\varepsilon,\lambda})\to (L^0(\mathcal{F}),\mathcal{T}_{\varepsilon,\lambda})$, namely $f'(x_n)(y)$ converges in $\mathcal{T}_{\varepsilon,\lambda}$ to $f'(x_0)(y)$ whenever a sequence $\{x_n, n\in N\}$ in $G$ converges in $\mathcal{T}_{\varepsilon,\lambda}$ to $x_0\in G$.
\end{definition}

\begin{remark}\label{Remark 3.12}
Definition \ref{Definition 3.11} is designed for a proper function defined on the whole random locally convex module $(E,\mathcal{P})$. Then for a function $f$ only defined on a $\mathcal{T}_{\varepsilon,\lambda}$--closed $L^0$--convex subset $G$ of $E$, when we speak of G\^{a}teaux--differentiability of $f$ on $G$ we mean that there exists an extension $\bar{f}$ of $f$ onto $E$ such that $\bar{f}$ is G\^{a}teaux--differentiable on $G($please note that Lemma \ref{Lemma 3.4} guarantees the existence of such an extension$)$. On the other hand, by Proposition \ref{Proposition 1.3} $\bigwedge\{\frac{f(x+\lambda y)-f(x)}{\lambda}~|~\lambda\in L^0_{++}(\mathcal{F})\}$ always exists (although the infimum may be any element of $\bar{L}^0(\mathcal{F})$), further it is easy to check for an $L^0$--convex function $f$ that the net $\{\frac{f(x+\lambda y)-f(x)}{\lambda}, \lambda\in L^0_{++}(\mathcal{F}) \}$ is nondecreasing in the sense : $\frac{f(x+\lambda_1 y)-f(x)}{\lambda_1}\leq \frac{f(x+\lambda_2 y)-f(x)}{\lambda_2}$ whenever $\lambda_1\leq\lambda_2($ here $L^0_{++}(\mathcal{F})$ is directed in the usual order$)$, so $\bigwedge\{\frac{f(x+\lambda y)-f(x)}{\lambda}~|~\lambda\in L^0_{++}(\mathcal{F}) \}=\lim_{\lambda\downarrow 0}\frac{f(x+\lambda y)-f(x)}{\lambda}$. In particular, when $f$ is G\^{a}teaux--differentiable at $x$, taking $\lambda=1$ and $v=y-x$ yields that $f(y)-f(x)\geq f'(x)(v)=f'(x)(y-x)$, namely $f'(x)$ is a subgradient of $f$ at $x$.
\end{remark}

\begin{lemma}\label{Lemma 3.13}
Let $(E,\mathcal{P})$ be a random locally convex module over $R$ with base $(\Omega,\mathcal{F},P)$ and $G$ a $\mathcal{T}_{\varepsilon,\lambda}$--closed $L^0$--convex subset of $E$. If $f : G\to L^0(\mathcal{F})$ is G\^{a}teaux--differentiable on $G$, then we have the following statements:
\begin{enumerate}[(1)]
\item $f$ is $L^0$--convex iff $f(y)\geq f(x)+f'(x)(y-x)$ for all $x,y\in G$;
\item $f$ is strictly $L^0$--convex iff $f(y)>f(x)+f'(x)(x-y)$ on $(x\neq y)$ for all $x,y\in G$;
\item If $f$ is $L^0$--convex, then $f'$ is monotone, namely $(f'(x)-f'(y))(x-y)\geq 0$ for all $x,y\in G$.
\end{enumerate}
\end{lemma}
\begin{proof}
Proofs are completely similar to those in classical cases $($as shown in Proposition 5.4 and Proposition 5.5 of \cite[Chapter I]{ET99} $)$, so are omitted.
\end{proof}

\begin{remark}\label{Remark 3.14}
In classical case, the converse of $(3)$ of Lemma \ref{Lemma 3.13} is also true, but in our current case, this is still open.
\end{remark}
\par
 Theorem \ref{Theorem 3.15} and \ref{Theorem 3.16} below are very interesting, whereas their proofs are also omitted since they are completely similar to proofs of Propositions 2.1 and 2.2 of \cite[Chapter II]{ET99}.

 \begin{theorem}\label{Theorem 3.15}
 Let $(E,\mathcal{P})$ and $G$ be the same as in Lemma \ref{Lemma 3.13}. If $f : G\to L^0(\mathcal{F})$ is a G\^{a}teaux--differentiable $L^0$--convex function with $f'$ $H$--continuous, then the following statements are equivalent to each other :
\begin{enumerate}[(1)]
\item $u\in G$ is such that $f(u)=\bigwedge \{f(v)~|~v\in G\}$;
\item $f'(u)(v-u)\geq 0$ for all $v\in G$;
\item $f'(v)(v-u)\geq 0$ for all $v\in G$.
\end{enumerate}
 \end{theorem}

\begin{theorem}\label{Theorem 3.16}
Let $(E,\mathcal{P})$ and $G$ be the same as in Theorem \ref{Theorem 3.15}. If $f_1:G\to L^0(\mathcal{F})$ is a G\^{a}teaux--differentiable $L^0$--convex function with $f'_1$ $H$--continuous, $f_2:G\to \bar{L}^0(\mathcal{F})$ a proper $L^0$--convex function and $f=f_1+f_2$, then the following statements are equivalent to each other :
\begin{enumerate}[(1)]
\item $u\in G$ is such that $f(u)=\bigwedge \{f(v)~|~v\in G\}$;
\item $f'_1(u)(v-u)+f_2(v)-f_2(u)\geq 0$ for all $v\in G$;
\item $f'_1(v)(v-u)+f_2(v)-f_2(u)\geq 0$ for all $v\in G$.
\end{enumerate}
\end{theorem}

\begin{corollary}\label{Corollary 3.17}[Proximity mappings]
Let $(E,(\cdot,\cdot))$ be a $\mathcal{T}_{\varepsilon,\lambda}$--complete random inner product module over $R$ with base $(\Omega,\mathcal{F},P)$. $f_1:E\to L^0(\mathcal{F})$ is defined by $f_1(u)=\frac{1}{2}\|u-x\|^2$ for all $u\in E(x$ is a fixed element of $E)$, $\varphi :E\to \bar{L}^0(\mathcal{F})$ is a proper, $\mathcal{T}_c$--lower semicontinuous and $L^0$--convex function and $f=f_1+ \varphi$. Then there exists a unique element $u\in E$ such that $f(u)=\bigwedge\{f(v)~|~v\in E\}$. This induces a mapping $Prox_{\varphi}:E\to E$ by $u=Prox_{\varphi}(x)$, called the proximity mapping with respect to $\varphi$.
\end{corollary}
\begin{proof}
Since $f_1$ is strictly $L^0$--convex and $\mathcal{T}_c$--lower semicontinuous, it is obvious that $f$ is strictly $L^0$--convex, $\mathcal{T}_c$--lower semicontinuous and proper. Further, since Fenchel--Moreau duality theorem has been established in \cite[Theorem 5.5]{GZZ15a}, $\varphi$ is bounded from below by a $\mathcal{T}_c$--continuous $L^0$--affine function $L$, which can be written as $L(v)=(v,w_0)+\alpha$ for some $w_0\in E$ and $\alpha\in L^0(\mathcal{F})($for all $v\in E)$ by Riesz's representation theorem \cite[Theorem 4.3]{GY96}. Now, one can easily verify that $f$ is coercive as in \cite[P.39]{ET99}. Finally, taking $G=E$ in Theorem \ref{Theorem 3.8} and noticing Remark \ref{Remark 3.9} yields that there exists a unique $u\in E$ such that $f(u)=\bigwedge\{f(v)~|~v\in E\}$
\end{proof}

\par
\begin{remark}\label{Remark 3.18}
Theorem \ref{Theorem 3.16} can be used to characterize $u=Prox_{\varphi}(x)$ in the following two equivalent ways :
\begin{enumerate}[(1)]
\item $(u-x,v-u)+\varphi(v)-\varphi(u)\geq 0$ for all $v\in E$;
\item $(v-x,v-u)+\varphi(v)-\varphi(u)\geq 0$ for all $v\in E$;
\end{enumerate}
\par
In particular, when $\varphi=$ the indicator function $\mathcal{X}_G$ of a $\mathcal{T}_{\varepsilon,\lambda}$--closed $L^0$--convex subset $G$ of $E$, $(1)$ and $(2)$ above become $(3)$ and $(4)$ below, respectively:
\begin{enumerate}[(3)]
\item $u\in G$ and $(u-x,v-u)\geq 0$ for all $v\in G$;
\end{enumerate}
\begin{enumerate}[(4)]
\item $u\in G$ and $(v-x,v-u)\geq 0$ for all $v\in G$;
\end{enumerate}
Thus $u$ is just the projection of $x$ onto $G$.

Finally, according to the above $(1)$ or $(2)$ and the stability of $\varphi$, one can easily verify that $Prox_{\varphi}(\cdot) : E \to E$ also has the stability, namely, $Prox_{\varphi}(\tilde{I}_{A}x_1 + \tilde{I}_{A^c}x_2) = \tilde{I}_{A} Prox_{\varphi}(x_1) + \tilde{I}_{A_c} Prox_{\varphi}(x_2)$ for all $x_1, x_2 \in E$ and $A \in \mathcal{F}$, further by noticing that $E$ has the countable concatenation property, then one can also see that $Prox_{\varphi}(\cdot)$ even has the $\sigma$--stability!
\end{remark}

\section{The existence of solutions of a general variational inequality}\label{Section 4}
The main result of this section is Theorem \ref{Theorem 4.1} below, throughout this section $(E,\|\cdot\|)$ always denotes a given random reflexive random normed module over $R$ with base $(\Omega,\mathcal{F},P)$.

\begin{theorem}\label{Theorem 4.1}
Let $\varphi:E\to \bar{L}^0(\mathcal{F})$ be a proper, $\mathcal{T}_c$--lower semicontinuous and $L^0$--convex function such that $dom(\varphi):=\{x\in E~|~\varphi(x)<+\infty$ on $\Omega\}$ is $\mathcal{T}_{\varepsilon,\lambda}$--closed, $f\in E^*$ and $M:E\to E^*$ a mapping satisfying the following conditions:
\begin{enumerate}[(M-1)]
\item $M$ is stable, namely $M(\tilde{I}_Ax+\tilde{I}_{A^c}y)=\tilde{I}_AM(x)+\tilde{I}_{A^c}M(y)$ for all $x,y\in E$ and $A\in \mathcal{F}$;
\item $M$ is weakly sequentially continuous over any finitely generated submodules $V$ of $E$, namely, for any fixed $y\in E, \{M(x_n)(y),n\in N\}$ converges in $\mathcal{T}_{\varepsilon,\lambda}$ to $M(x_0)(y)$ $($remind: $\mathcal{T}_{\varepsilon,\lambda}$ on $L^0(\mathcal{F})$ is exactly the topology of convergence in probability$)$ whenever $\{x_n,n\in N\}$ in $V$ converges in $\mathcal{T}_{\varepsilon,\lambda}$ to $x_0\in V$;
\item $M$ is monotone, namely $(M(u)-M(v))(u-v)\geq 0$ for all $u,v\in E$;
\item There exists some $v_0\in dom(\varphi)$ such that $\{\frac{M(v_n)(v_n-v_0)+\varphi(v_n)}{\|v_n\|},n\in N\}$ converges almost surely to $+\infty$ on $A$ whenever $\{\|v_n\|,n\in N\}$ converges almost surely to $+\infty$ on some $A\in \mathcal{F}$ with $P(A)>0$.
\end{enumerate}
Then for any given $f\in E^*$ there exists at least one $u\in E$ such that the following is satisfied :
\begin{enumerate}[(4.1)]
\item $(M(u)-f)(v-u)+\varphi(v)-\varphi(u)\geq 0$ for all $v\in E$.
\end{enumerate}
\end{theorem}

\par
Before we give the proof of Theorem \ref{Theorem 4.1}, let us first recall: $E$ is random reflexive, it must be $\mathcal{T}_{\varepsilon,\lambda}$--complete and hence also have the countable concatenation property, from this one can easily see that $(M-1)$ also implies $M$ is both local $($namely $\tilde{I}_AM(v)=\tilde{I}_AM(\tilde{I}_Av)$ for all $v\in E$ and $A\in \mathcal{F})$ and $\sigma$--stable $($ namely $M(\Sigma^{\infty}_{n=1}\tilde{I}_{A_n}v_n)=\Sigma^{\infty}_{n=1}\tilde{I}_{A_n}M(v_n)$ for all sequence $\{v_n,n\in N\}$ and all countable partitions $\{A_n,n\in N\}$ of $\Omega$ to $\mathcal{F})$.
\par
Proof of Theorem \ref{Theorem 4.1} is more involved than its classical counterpart--Theorem 3.1 of \cite[Chapter II]{ET99}, thus it needs Lemma \ref{Lemma 4.3}, Lemma \ref{Lemma 4.4} and Lemma \ref{Lemma 4.5} below. Before the lemmas are given, let us restate a nice fixed point theorem duo to Drapeau.et.al's\cite{DKKS13} in a slightly more general form since $K$ in \cite[Proposition 3.1]{DKKS13} is exactly $\mathcal{T}_{\varepsilon,\lambda}$--closed and the mapping $f$ in \cite[Proposition 3.1]{DKKS13} is only needed to be $\mathcal{T}_{\varepsilon,\lambda}$--continuous , besides, since $K$ has the countable concatenation property, one can easily see that the stability of $f$ has also implied the $\sigma$--stability of $f$.

\begin{proposition}\label{Proposition 4.2}\cite[Proposition 3.1]{DKKS13}
Let $K$ be an $L^0$--convex, $\mathcal{T}_{\varepsilon,\lambda}$--closed and almost surely bounded subset of $L^0(\mathcal{F},R^d)$ $($or, $(L^0)^d)$ and $f: K\to K$ a stable and $\mathcal{T}_{\varepsilon,\lambda}$--continuous mapping. Then $f$ has a fixed point.
\end{proposition}

\begin{lemma}\label{Lemma 4.3}
Theorem \ref{Theorem 4.1} is true if $E$ is finitely generated and $dom(\varphi)$ is almost surely bounded.
\end{lemma}
\begin{proof}
Since $E$ is finitely generated, Corollary \ref{Corollary 2.25} guarantees that $(E,\|\cdot\|)$ can be identified with a $\mathcal{T}_{\varepsilon,\lambda}$--closed submodule of $\mathcal{T}_{\varepsilon,\lambda}$--complete random inner product module $( L^0(\mathcal{F},R^d), (\cdot,\cdot) )$ $($for some positive integer $d)$ in  the sense of topological module isomorphism, in particalar $E$ and $E^*$ are identified. Then if $u$ is a solution for $(4.1)$, we have: $(M(u)-f,v-u)+\varphi(v)-\varphi(u)\geq 0$ for all $v\in E$, or equivalently, $(u-(u+f-M(u)),v-u)+\varphi(v)-\varphi(u)\geq 0$ for all $v\in E$, which amounts to saying $u=prox_{\varphi}(u+f-M(u))$.
\par
Since $E\subset L^0(\mathcal{F},R^d)$, it is easy to see that $M$ is continuous from $(E,\mathcal{T}_{\varepsilon,\lambda})$ to $(E,\mathcal{T}_{\varepsilon,\lambda})$, namely $\mathcal{T}_{\varepsilon,\lambda}$--continuous, so is the mapping sending $u\in E$ to $u+f-M(u)$.
\par
Now, we prove that $prox_{\varphi}(\cdot): E\to dom(\varphi)\subset E$ is $\mathcal{T}_{\varepsilon,\lambda}$--continuous $($moreover, Lipschitz continuous$)$: in fact, let $x_1$ and $x_2$ belong to $E$ and $u_1=prox_{\varphi}(x_1)$ and $u_2=prox_{\varphi}(x_2)$, then by the property given in Remark \ref{Remark 3.18} one can have: $(u_1-x_1,u_2-u_1)+\varphi(u_2)-\varphi(u_1)\geq 0$ and $(u_2-x_2,u_1-u_2)+\varphi(u_1)-\varphi(u_2)\geq 0$, and by addition: $\|u_1-u_2\|^2\leq (x_1-x_2,u_1-u_2)\leq \|x_1-x_2\|\|u_1-u_2\|$, namely $\|u_1-u_2\|\leq \|x_1-x_2\|$.
\par
Thus the composite mapping $T: dom(\varphi)\to dom(\varphi)$ given by $T(v)=prox_{\varphi}(v+f-M(v))$ for each $v\in dom(\varphi)$ is $\mathcal{T}_{\varepsilon,\lambda}$--continuous. Further, since both $M$ and $Prox_{\varphi}(\cdot)$ are stable, $T$ is also stable, then by Proposition \ref{Proposition 4.2} there exists $u\in dom(\varphi)$ such that $u=prox_{\varphi}(u+f-M(u))$, $u$ is just desired.
\end{proof}

\par
The case for $L^0$--modules is much more complicated than that for ordinary linear spaces, to overcome the complications we prove the following key lemma which is freguently employed in the process of Proof of Theorem \ref{Theorem 4.1}.

\begin{lemma}\label{Lemma 4.4}
Let $E, M, f$ and $\varphi$ be the same as in Theorem \ref{Theorem 4.1}. If $G$ is an $L^0$--convex subset of $E$ and $u\in G$, then the following two conditions are equivalent to each other:
\begin{enumerate}[(i)]
\item $(M(u)-f)(v-u)+\varphi(v)-\varphi(u)\geq 0$ for all $v\in G$.
\item $(M(v)-f)(v-u)+\varphi(v)-\varphi(u)\geq 0$ for all $v\in G$.
\end{enumerate}
\end{lemma}
\begin{proof}
$(i)\Rightarrow(ii)$ is clear by the monotone property of $M$.
\par
Now, if $(ii)$ holds, let $v=(1-\lambda)u+\lambda w$ for any $w\in G$ and $\lambda \in L^0_{++}(\mathcal{F})$ such that $0<\lambda <1$ on $\Omega$ $($namely $P\{\omega\in\Omega~|~0<\lambda(\omega) <1\}=1)$, then $v\in G$ and $(M((1-\lambda)u+\lambda w)-f)(\lambda w-\lambda u)+\varphi((1-\lambda)u+\lambda w)-\varphi(u)\geq 0$, the $L^0$--convexity of $\varphi$ yields that $\lambda (M((1-\lambda)u+\lambda w)-f)(w-u)+\lambda (\varphi (w)-\varphi(u))\geq 0$, namely $(M((1-\lambda)u+\lambda w)-f)(w-u)+\varphi (w)-\varphi(u)\geq 0$. Arbitrarily taking a sequence $\{\lambda_n,n\in N\}$ such that $0<\lambda_n<1$ on $\Omega$ and $\lambda_n\downarrow 0$, then $(M((1-\lambda_n)u+\lambda_n w)-f)(w-u)+\varphi (w)-\varphi(u)\geq 0$. By applying $(M-2)$ to the submodule generated by $\{u,w\}$, one can have that $(M(u)-f)(w-u)+\varphi(w)-\varphi(u)\geq 0$. Since $w$ is arbitrary, $(i)$ holds.
\end{proof}

\begin{lemma}\label{Lemma 4.5}
Theorem \ref{Theorem 4.1} is valid if $E$ is finitely generated.
\end{lemma}
\begin{proof}
For any $r\in L^0_{++}(\mathcal{F})$, let $\varphi_r=\varphi+\mathcal{X}_{U(r)}$, where $U(r)=\{u\in E~|~\|u\|\leq r\}$ and $\mathcal{X}_{U(r)}$ is the indicator function of the $\mathcal{T}_{\varepsilon,\lambda}$--closed $L^0$--convex subset $U(r)$, see the comments following Lemma \ref{Lemma 3.4} for the notion of an indicator function. Since $dom(\varphi_r)=dom(\varphi)\cap U(r), \varphi_r$ is $L^0$--convex, proper and $\mathcal{T}_c$--lower semicontinuous for $r$ sufficiently large.

\par
Lemma \ref{Lemma 4.3} implies the existence of $u_r$ such that $\|u_r\|\leq r$ and satisfies the following:
\begin{enumerate}[(1)]
\item $(M(u_r)-f,v-u_r)+\varphi_r(v)-\varphi_r(u_r)\geq 0,\forall v\in E$.
\end{enumerate}

\par
Now, for $r$ sufficiently large such that $r\geq \|v_0\|$, we can put $v=v_0$ in $(1)$ and find that $(1)$ becomes the following:
\begin{enumerate}[(2)]
\item $(M(u_r)-f,v_0-u_r)+\varphi(v_0)-\varphi(u_r)\geq 0$.
\end{enumerate}
\par
At the present time, we can assert that the net $\{u_r,r\geq\|v_0\|\}$ is alomst surely bounded. Otherwise, there exist some $A\in\mathcal{F}$ of positive probability and a sequence $\{r_n,n\in N\}$ such that $r_n\in L^0_{++}(\mathcal{F})$ and $r_n\geq\|v_0\|$ for each $n\in N$ and such that $\bigvee_{n\geq 1}\|u_{r_n}\|=\bigvee_{r\geq\|v_0\|}\|u_r\|=+\infty$ on $A$. By the same technique as in Proof of Theorem \ref{Theorem 3.8}, there exists a sequence $\{u^*_{r_n},n\in N\}$ in $E$ such that the following are satisfied: \begin{enumerate}[(3)]
\item for each $n\in N$ there exists a finite partition $\{B_k,1\leq k\leq n\}$ of $\Omega$ to $\mathcal{F}$ such that $u^*_{r_n}=\sum^n_{k=1}\tilde{I}_{B_k}\cdot u_{r_k}$;
\end{enumerate}
\begin{enumerate}[(4)]
\item $\|u^*_{r_n}\|=\bigvee^n_{k=1}\|u_{r_k}\|$ for each $n\in N$.
\end{enumerate}
\par
Since $(2)$ holds for each $r=r_n$ and both $M$ and $\varphi$ are stable, $(2)$ still holds for each $u^*_{r_n}$, namely $(M(u^*_{r_n})-f,v_0-u^*_{r_n})+\varphi(v_0)-\varphi(u^*_{r_n})\geq 0$ for each $n\in N$, then one can see:
\begin{enumerate}[(5)]
\item $\frac{1}{\|u^*_{r_n}\|}\{(M(u^*_{r_n}),u^*_{r_n}-v_0)+\varphi(u^*_{r_n})\}\leq \frac{1}{\|u^*_{r_n}\|}\{-(f,v_0-u^*_{r_n})+\varphi(v_0)\}$ on $A$.
\end{enumerate}
\par
Since the right--hand side of $(5)\leq \frac{|f(v_0)-\varphi(v_0)|}{\|u^*_{r_n}\|}+\|f\|$ on $A$, when $n$ tends to $\infty$ , $\frac{|f(v_0)-\varphi(v_0)|}{\|u^*_{r_n}\|}+\|f\|$ tends to $\|f\|$ on $A$, but the left--hand side of $(5)$ tends to $+\infty$ on $A$ as $n$ tends to $\infty$ by the condition $(M-4)$, which is a contradiction.
\par
Now that we have proved $\{u_r,r\geq\|v_0\|\}$ to be almost surely bounded, there exists $\xi_0\in L^0_{++}(\mathcal{F})$ such that $\|u_r\|\leq\xi_0$ for all $r\geq\|v_0\|$. We can, without loss of generality, assume that $\xi_0\geq\|v_0\|$. Observing that $\mathcal{U}=\{U(\frac{1}{r})~|~r\in L^0_{++}(\mathcal{F})$ and $r\geq\xi_0\}$ is a local base at $\theta$ of $\mathcal{T}_c$, by the $L^0$--convex compactness of $U(\xi_0)$ and the reasoning used in Theorem \ref{Theorem 2.16} there exists a net $\{y_r,r\geq\xi_0\}$ in $U(\xi_0)$ convergent in $\mathcal{T}_c$ to some $u\in U(\xi_0)$ such that :
\begin{enumerate}[(6)]
\item $y_r\in(u+U(\frac{1}{r}))\cap H_{cc}(conv_{L^0}\{u_{r'},r'\geq r\})$ for each $r\geq\xi_0$.
\end{enumerate}
\par
Now, for any given $r_0\geq\xi_0$, we will prove the following assertion:
\begin{enumerate}[(7)]
\item $(M(v)-f,v-u)+\varphi(v)-\varphi(u)\geq 0$ for all $v\in U(r_0)$.
\end{enumerate}
\par
First,when $r\geq r_0$, $(1)$ implies the following:
\begin{enumerate}[(8)]
\item $(M(u_r)-f,v-u_r)+\varphi(v)-\varphi(u_r)\geq 0$ for all $v\in U(r_0)$.
\end{enumerate}
\par
By applying Lemma \ref{Lemma 4.4} to $u=u_r$ and $G=U(r_0)$, then $(8)$ amounts to the following:
\begin{enumerate}[(9)]
\item $(M(v)-f,v-u_r)+\varphi(v)-\varphi(u_r)\geq 0$ for all $v\in U(r_0)$.
\end{enumerate}
\par
For any fixed $v\in U(r_0)$, let $L(v)=\{w\in U(\xi_0)~|~\varphi(w)+(M(v)-f,w)\leq (M(v)-f,v)+\varphi(v)\}$, then it is obvious that $L(v)$ is $\mathcal{T}_c$--closed and $L^0$--convex and has the countable concatenation property. Further, since, for each $r\geq r_0$, $u_r\in L(v)$, then $y_r$ eventually falls into $L(v)$ for any $r\geq r_0$, so that $u$ also belongs to $L(v)$, namely $u$ satisfies $(7)$.
\par
Again applying Lemma \ref{Lemma 4.4} to $G=U(r_0)$ yields:
\begin{enumerate}[(10)]
\item $(M(u)-f,v-u)+\varphi(v)-\varphi(u)\geq 0$ for all $v\in U(r_0)$.
\end{enumerate}
\par
Finally, since $r_0(\geq\xi_0)$ is arbitrary, one can see the following assertion:
\begin{enumerate}[(11)]
\item $(M(u)-f,v-u)+\varphi(v)-\varphi(u)\geq 0$ for all $v\in E$.
\end{enumerate}
\end{proof}
\par
With Lemma \ref{Lemma 4.4} and the method of Proof of Lemma \ref{Lemma 4.5}, now we can return to Proof of Theorem \ref{Theorem 4.1}, a careful reader even can observe that the method provided here is simpler than that of Proof of its classical prototype--Theorem 3.1 of \cite[Chapter II]{ET99}.

\begin{proof}[Proof of Theorem \ref{Theorem 4.1}]
Let $\mathcal{V}$ be the family of finitely generated submodules of $E$ which contains $v_0, \mathcal{V}$ is directed by the usual inclusion relation, namely $V_1\leq V_2$ iff $V_1\subset V_2$.
On the other hand, since $\mathcal{V}$ does not necessarily have a largest element, we can not use a single random inner product $(\cdot,\cdot)$ but return to the usual pairing between $E^*$ and $E: \langle v^*,u\rangle$ or $v^*(u)$ for any $v^*\in E^*$ and $u\in E$.
\par
For each $V\in \mathcal{V}$, by Lemma \ref{Lemma 4.5} there exists $u_V\in V$ such that
\begin{enumerate}[(1)]
\item $(M(u_V)-f)(v-u_V)+\varphi(v)-\varphi(u_V)\geq 0$ for all $v\in V$.
\end{enumerate}
Applying Lemma \ref{Lemma 4.4} to $u=u_V$ and $G=V$, one can easily see that $(1)$ is equivalent to the following:
\begin{enumerate}[(2)]
\item $(M(v)-f)(v-u_V)+\varphi(v)-\varphi(u_V)\geq 0$ for all $v\in V$.
\end{enumerate}
\par
By noticing $v_0\in V$ and reasoning as in the proof of Lemma \ref{Lemma 4.5}, one can easily see that $\{u_V,V\in \mathcal{V}\}$ must be almost surely bounded, and thus there exists $\xi_0\in L^0_{++}(\mathcal{F})$ such that $\|u_V\|\leq\xi_0$ for all $V\in \mathcal{V}$.
\par
Since $E$ is random reflexive, $U(\xi_0)$ is $L^0$--convexly compact, similarly to Proof of Lemma \ref{Lemma 4.5}, applying Theorem \ref{Theorem 2.16} can produce a net $\{y_{(V,\xi)},(V,\xi)\in \mathcal{V}\times L^0_{++}(\mathcal{F})\}$ convergent in $\mathcal{T}_c$ to some $u\in U(\xi_0)$ such that
\begin{enumerate}[(3)]
\item $y_{(V,\xi)}\in (u+U(\frac{1}{\xi}))\cap H_{cc}(conv_{L^0}\{u_{V'},V'\geq V\})$ for all $(V,\xi)\in \mathcal{V}\times L^0_{++}(\mathcal{F}) ($ where $L^0_{++}(\mathcal{F})$ is directed by the usual order relation $\leq$ on $L^0_{++}(\mathcal{F}) )$.
\end{enumerate}
Now, we will assert that the following inequality always holds for an arbitrarily fixed $V\in \mathcal{V}$:
\begin{enumerate}[(4)]
\item $(M(v)-f)(v-u)+\varphi(v)-\varphi(u)\geq 0$ for all $v\in V$.
\end{enumerate}
\par
First, when $V'($in $\mathcal{V})\geq V$, one can easily see that the following inequality, can, of course, be seen from $(2)$:
\begin{enumerate}[(5)]
\item $(M(v)-f)(v-u_{V'})+\varphi(v)-\varphi(u_{V'})\geq 0$ for all $v\in V$.
\end{enumerate}
Similarly to Proof of Lemma \ref{Lemma 4.5}, for any fixed $v\in V$, let $L(v)=\{w\in E~|~\varphi(w)+(M(v)-f)(w)\leq\varphi(v)+(M(v)-f)(v)\}$, then $L(v)$ is $\mathcal{T}_c$--closed and $L^0$--convex and has the countable concatenation property. Further, since $u_{V'}\in L(v)$ for any $V'\geq V$,each $y_{(V,\xi)}\in L(v)$, then we eventually have that $u\in L(v)$, namely the following inequality is valid:
\begin{enumerate}[(6)]
\item $(M(v)-f)(v-u)+\varphi(v)-\varphi(u)\geq 0$ for all $v\in V$.
\end{enumerate}
\par
Since $V$ is arbitrarily chosen from $\mathcal{V}$, for each $v\in E$, one can always choose a $V$ from $\mathcal{V}$ such that $v\in V$, so that we can have:
\begin{enumerate}[(7)]
\item $(M(v)-f)(v-u)+\varphi(v)-\varphi(u)\geq 0$ for all $v\in E$.
\end{enumerate}
\par
Finally, by Lemma \ref{Lemma 4.4}, $(7)$ amounts to the following:
\begin{enumerate}[(8)]
\item $(M(u)-f)(v-u)+\varphi(v)-\varphi(u)\geq 0$ for all $v\in E$.
\end{enumerate}
This completes the proof of Theorem.
\end{proof}

\begin{remark}\label{Remark 4.6}
If $M$ in Theorem \ref{Theorem 4.1} is moreover strictly monotone: $(M(w)-M(v))(w-v)>0$ on $(w\neq v)$ for any $w,v\in E$, then $u$ is unique.
\end{remark}

\begin{proposition}\label{Proposition 4.7}
If $G$ is a $\mathcal{T}_{\varepsilon,\lambda}$--closed $L^0$--convex subset of $E$ and $M: E\to E^*$ satisfies the same hypotheses as in Theorem \ref{Theorem 4.1}, then for any given $f\in E^*$ there exists $u\in G$ such that $(M(u)-f)(v-u)\geq 0$ for all $v\in G$.
\end{proposition}
\begin{proof}
We can, without loss of generality, assume $\theta\in G$ $($otherwise, let $u_0\in G, G_0=G-u_0$ and $M_0: E\to E^*$ be defined by $M_0(v)=M(v+u_0)$ for all $v\in E$ $)$, define $\varphi: E\to \bar{L}^0(\mathcal{F})$ by $\varphi=\mathcal{X}_G$, then by Theorem \ref{Theorem 4.1} there exists $u\in G$ such that $(M(u)-f)(v-u)\geq 0$ for all $v\in G$.
\end{proof}

\begin{corollary}\label{Corollary 4.8}
Let $M$ satisfy the same hypotheses as in Theorem \ref{Theorem 4.1}, then for any $f\in E^*$ there exists $u\in E$ such that $M(u)=f$.
\end{corollary}
\begin{proof}
Taking $G=E$ in Proposition \ref{Proposition 4.7} will complete the proof.
\end{proof}

\begin{proposition}\label{Proposition 4.9}
Let $E$ be a $\mathcal{T}_{\varepsilon,\lambda}$--complete random inner product module and $M: E\to E$ a $\mathcal{T}_{\varepsilon,\lambda}$--continuous module homomorphism such that, for some $\alpha\in L^0_{++}(\mathcal{F}), (M(v),v)\geq\alpha\|v\|^2$ for all $v\in E$, then for any given $f\in E$ there exists a unique $u\in E$ such that $M(u)=f$.
\end{proposition}
\begin{proof}
$E$ and $E^*$ are identified, it is easy to check that $M$ satisfies all the conditions as in Theorem \ref{Theorem 4.1}, then by Corollary \ref{Corollary 4.8} there exists a unique $u\in E$ such that $M(u)=f$.
\end{proof}

\section*{Acknowledgement(s)}
The authors are supported by the NNSF of China No.11571369 and No.11701531. The authors would also like to thank Professors Hongkun Xu and George Yuan for some valuable suggestions which considerably improve the readability of this paper.


\bibliography{amsplain}

\end{document}